\documentclass[reqno]{amsart}
\usepackage{amssymb}
\usepackage{graphicx}

\usepackage[usenames, dvipsnames]{color}
\usepackage{verbatim}
\usepackage{mathrsfs}

\numberwithin{equation}{section}

\newtheorem{theorem}{Theorem}[section]

\newtheorem{lemma}[theorem]{Lemma}
\newtheorem{prop}[theorem]{Proposition}

\theoremstyle{definition}
\newtheorem{remark}[theorem]{Remark}

\theoremstyle{definition}

\theoremstyle{definition}

\makeatletter
\def\dashint{\operatorname%
{\,\,\text{\bf-}\kern-.98em\DOTSI\intop\ilimits@\!\!}}
\makeatother

\def\\det{\text{det}}

\def\.5{\frac{1}{2}}

\newcommand{\RN}[1]{%
  \textup{\uppercase\expandafter{\romannumeral#1}}%
}

\renewcommand{\epsilon}{\varepsilon}

\newcounter{marnote}

\begin{document}
\title[A Liouville theorem for periodic solution]{A Liouville theorem for ancient solutions of the parabolic Monge-Ampère equation with periodic data}
\author[K. Yan]{Kui Yan}
\address[K. Yan]{School of Mathematical Sciences, Beijing Normal University, Laboratory of Mathematics and Complex Systems, Ministry of Education, Beijing 100875, China.}
\email{202231130028@mail.bnu.edu.cn}

\author[J. Bao]{Jiguang Bao\textsuperscript{*}}
\address[J. Bao]{School of Mathematical Sciences, Beijing Normal University, Laboratory of Mathematics and Complex Systems, Ministry of Education, Beijing 100875, China.}
\email{jgbao@bnu.edu.cn}

\footnotetext{*corresponding author}


\date{\today} 

\begin{abstract}
This article is concerned with the parabolic Monge-Ampère equation $-u_t\det D_x^2u=f$, where $f=f_1(x)f_2(t)$ and $f_1,f_2$ are positive periodic functions. We prove that any classical parabolically convex ancient solution $u$ must be of the form $-\tau t+p(x)+v(x,t)$, where $\tau$ is a positive constant, $p(x)$ is a convex quadratic polynomial, and $v$ inherits both the spatial and temporal periodicity from $f$. This work extends previous contributions by Caffarelli-Li \cite{cl04} on periodic frameworks for the elliptic Monge-Ampère equations, and generalizes Zhang-Bao \cite{zb18}'s Liouville theorem for $f_2\equiv1$ in parabolic case.

\end{abstract}
\maketitle

\section{Introduction and main results} \label{1}
The aim of this paper is the classification of all $C^{2,1}$ parabolically convex ancient solutions of the parabolic Monge-Ampère equation
\begin{equation}\label{pma1}
-u_t\det D_x^2u=f(x,t)\quad \text{in} \quad\mathbb{R}^{n+1}_-.
\end{equation}
Here $\mathbb{R}^{n+1}_-:=\mathbb{R}^n\times\left(-\infty,0\right]$, the notation $u_t$ designates the temporal derivative, $D_x^2u$ corresponds to the Hessian acting on the spatial variables and $f(x,t)$ is a known periodic function. $u(x,t)$ is called parabolically convex if it is convex in $x$ for each fixed $t$ and nonincreasing in $t$ for each fixed $x$. We denote $D_x^i D_t^j u$ as $u$'s mixed derivative, which is $i\text{-th}$ order in $x$ and $j\text{-th}$ order in $t$. We say that $u$ is $C^{2 k, k}(k\in\mathbb{N}^+)$ if $D_x^i D_t^j u$ is continonus for indices satisfying $i,j\in\mathbb{N}$ and $i+2 j \leq 2 k$. 

The equation \eqref{pma1} was first introduced by Krylov \cite{k76}. This type of equation arises in stochastic control theory \cite{k81a,k81b}, surface evolution‌ governed by Gauss-Kronecker curvature \cite{f74,t85} and maximum principle for parabolic equations \cite{k81}. 

Under the time-derivative constraint
\begin{equation}\label{u_t}
m_1\le -u_t\le m_2\quad \text{in}\quad \mathbb{R}^{n+1}_-
\end{equation}
for constants $m_1,m_2>0$, Gutiérrez-Huang \cite{gh98} proved that any $C^{4,2}$ parabolically convex ancient solution to \eqref{pma1} with $f\equiv1$ must be of form $-\tau t+p(x)$, where $\tau>0$ and $p(x)$ is a convex quadratic polynomial. Xiong-Bao \cite{xb11} extended Guti\'errez-Huang's results to more general equations. Zhang-Bao-Wang \cite{zbw18} established that when the support of $f-1$ is bounded and $n\ge3$, $C^{2,1}$ parabolically convex ancient solution tends to $-\tau t+p(x)$ for certain decay rate at infinity. Wang-Bao \cite{wb15} obtained a similar result for another kind of parabolic Monge-Ampère equation and $n\ge4$. 

Zhang-Bao \cite{zb18} extended Gutiérrez-Huang's result to the case where $f=f(x)$ is a positive, periodic and Hölder continuous function. Their analysis revealed that any $C^{2,1}$ parabolically convex ancient solution must admit the form $-\tau t+p(x)+v(x)$, where $v(x)$ is a function sharing the same periodicity as $f(x)$.

In contrast‌ to previous studies, this paper addresses a broader framework where the right-hand side of the equation \eqref{pma1} assumes a separable structure $f(x,t)=f_1(x)f_2(t)$. More precisely, for $\alpha\in(0,1)$, we impose the following hypothesis: $f_1\in C^\alpha\left(\mathbb{R}^n\right)$, $f_2\in C^{\frac{\alpha}{2}}\left(-\infty,0\right]$ are positive functions satisfying that
\begin{equation}\label{f1f2}
f_1(x+a_ie_i)=f_1(x) \quad \text{and} \quad f_2(t-a_0)=f_2(t) \quad (x,t)\in\mathbb{R}^{n+1}_-,
\end{equation}
where $a_0,a_1,\cdots, a_n>0$ are constants and $e_1:=(1,0,\cdots,0)$, $\cdots$, $e_n:=(0,\cdots,0,1)$ are standard basis vectors in $\mathbb{R}^n$.

Our main result is as follows.

\begin{theorem}\label{mainthm1}
Let $n\ge1$ and $u\in C^{2,1}\left(\mathbb{R}^{n+1}_-\right)$ be a parabolically convex ancient solution to \eqref{pma1} with \eqref{u_t}, where $f:=f_1f_2$ and $f_1,f_2$ satisfy \eqref{f1f2}. Then there exist a constant $\tau>0$, an $n\times n$ symmetric positive definite matrix $A$ with 
\[
\tau\det A=-\mkern-19mu\int_{\left(\prod\limits_{1\le i\le n}[0,a_i]\right)\times[-a_0,0]}f(x,t)dxdt,
\]
and a vector $b\in\mathbb{R}^n$, such that 
\[
v(x,t):=u(x,t)-\left(-\tau t+\frac{1}{2}x'Ax+b\cdot x\right)
\] 
is $a_i$-periodic in $x_i$ variable for each $1\le i\le n$ and $a_0$-periodic in $t$ variable.
\end{theorem}

\begin{remark}
The results of Gutiérrez-Huang \cite{gh98} and Zhang-Bao \cite{zb18} can be directly derived as special cases of Theorem \ref{mainthm1}. Specifically, our result generalizes Zhang-Bao's work by extending the case of $f_2\equiv1$ to any positive periodic Hölder continous function $f_2$ defined on $\left(-\infty,0\right]$.
\end{remark}

\begin{remark}
Theorem \ref{mainthm1} remains valid even when the orthonormal basis vectors $e_1,\cdots,e_n$ are replaced by arbitrary linearly independent vectors $\epsilon_1,\cdots,\epsilon_n$. Specifically, if $f_1,f_2$ satisfy the generalized periodicity conditions 
\[
f_1(x+\epsilon_i)=f_1(x)(1\le i\le n) \quad \text{and}\quad f_2(t-a_0)=f_2(t)\quad \text{for} \quad (x,t)\in \mathbb{R}^{n+1}_-,
\] 
then the function $v$ inherits the same periodicity:
\[
v(x+k\epsilon_i,t-la_0)=v(x,t)\quad\text{for}\quad (x,t)\in \mathbb{R}^{n+1}_-,\quad k,l\in\mathbb{N},\quad 1\le i\le n.
\]
\end{remark}

\begin{remark}
Without loss of generality, we may assume that $a_i=1$ for all $1\le i\le n$ and $a_0=1$. Moreover, we can normalize the functions such that
\[
-\mkern-19mu\int_{[0,1]^n\times[-1,0]}f(x,t)dxdt=-\mkern-19mu\int_{[0,1]^n}f_1(x)dx=-\mkern-19mu\int_{[-1,0]}f_2(t)dt=1.
\]
Meanwhile, we may also assume that the function $f$ satisfies the uniform bounds
\[
\lambda\le f\le \Lambda,
\]
where $\lambda,\Lambda$ are positive constants.
\end{remark}

The elliptic Monge-Ampère equation plays a central role in several problems, and has been applied to various fields in analysis and geometry. The foundational result by from Jörgens-Calabi-Pogorelov \cite{j54,c58,p72} asserts that $C^2$ convex entire solution of $\det D_x^2u=1$ in $\mathbb{R}^n$ must be a quadratic polynomial. Cheng-Yau \cite{cy86} and  Jost-Xin \cite{jx01} established different proofs. Caffarelli \cite{c95} generalized the result to viscosity solutions. Significant attention has also been paid to non-constant right-hand sides. There are many studies covering cases where the support of $f-1$ is bounded \cite{fmm99,cl03,h20,lb1,lb2} or $f$ tends to $1$ at infinity \cite{blz15,lb3,lb4,lb5}.

The elliptic Monge-Ampère equation $\det D_x^2u=f$ in $\mathbb{R}^n$ with periodic $f$ can be found in various fields such as homogenization theory, optimal transportation, vorticity arrays, differential geometry, etc. Caffarelli-Li \cite{cl04} obtained that when $f\in C^\alpha\left(\mathbb{R}^n\right)$ is a positive periodic function, $u$ decomposes into a quadratic polynomial plus a periodic function. Lu-Li \cite{ll22} proved it when $f$ is merely positive and bounded. Teixeira-Zhang \cite{tz16} obtained that for $n\ge3$, if $f\in C^{1,\alpha}\left(\mathbb{R}^n\right)$ is a perturbation of a periodic function, $u$ minus a parabola tends to a periodic function at infinity. 

This paper is structured as follows: At the conclusion of Section \ref{1}, we introduce key notations. In Section \ref{1'}, we derive some fine interior estimates for solutions to the parabolic Monge-Ampère equation in bounded domains. In Section 3 we establish a homogenization estimate, which plays a central role in our analysis. This estimate, combined with Caffarelli-Li's nonlinear perturbation argument, enables us to derive the quadratic behavior of $u$ at infinity in Section \ref{3}. Subsequently, in Section \ref{3'}, we obtain the $C^0$ estimate for $D_x^2 u $. Then in Section \ref{4}, we characterize $u$'s second-order difference quotients in $x$ variable and first-order quotients in $t$ variable. Finally in Section \ref{5} we present the proof of Theorem \ref{mainthm1}.

A subset $\mathcal{D}$ in $\mathbb{R}^{n+1}_-$ is called a bowl-shaped domain if its time slices
\[
\mathcal{D}(t):=\left\{x \in \mathbb{R}^n:(x, t) \in \mathcal{D}\right\}
\] 
are convex for each $t\le0$ and $\mathcal{D}\left(t_1\right) \subset \mathcal{D}\left(t_2\right)$ whenever $t_1 \leq t_2$ with $D\left(t_2\right)\neq\emptyset$. The parabolic boundary of a bounded bowl-shaped domain $\mathcal{D}$ is given by
\[
\partial_p \mathcal{D}:=\left(\overline{\mathcal{D}(t_0)}\times\{t_0\}\right)\cup  \bigcup_{t \le0}(\partial \left(\mathcal{D}(t)\right) \times\{t\}),
\]
where $t_0:=\inf\{t\le0:(x,t)\in \mathcal{D}\}$. A canonical example is the parabolic ball: for $\delta>0$ and $(x_0,t_0)\in\mathbb{R}^{n+1}_-$,
\[
E_{\delta}(x_0,t_0):=\{(x,t)\in\mathbb{R}^{n+1}_-:|x-x_0|^2-(t-t_0)< \delta^2,\quad t\le t_0 \}.
\]
We abbreviate $E_{\delta}:=E_{\delta}\left(0_n,0\right)$, where $0_n$ denotes the origin in $\mathbb{R}^n$. 

We denote the parabolic distance between $X=(x,t),Y=(y,s)\in \mathbb{R}^{n+1}_-$ by 
\[
\left|X-Y\right|_p:=\left(|x-y|^2+|t-s|\right)^{\frac{1}{2}}. 
\]
The parabolic distance from $Z\in \mathbb{R}^{n+1}_-$ to a set $\mathcal{D}\subset\mathbb{R}^{n+1}_-$ is defined by
\[
\text{dist}_p\left(Z,\mathcal{D}\right):=\inf\left\{\left|Z-Z_1\right|_p: Z_1\in\mathcal{D} \right\}.
\]

For simplicity, we adopt the following shorthand for partial derivatives:
\[
u_i:=D_{x_i}u,\quad u_{ij}:=D_{x_ix_j}u,\quad u_{ijk}:=D_{x_ix_jx_k}u,\quad u_{ijkl}:=D_{x_ix_jx_kx_l}u,
\]
\[
u_{it}:=D_{x_i}D_tu,\quad u_{ijt}:=D_{x_ix_j}D_tu.
\]
\quad

\section{Interior estimates when $f=1$ in bounded domains} \label{1'}

We begin by establishing pointwise lower and upper bounds for solutions in terms of the parabolic distance to the parabolic boundary.

\begin{lemma}\label{lemma}
Let $\mathcal{D}\subset\mathbb{R}^{n+1}_-$ be a bounded bowl-shaped domain and $u\in C^{2,1}\left(\mathcal{D}\right)\cap C^0\left(\overline{\mathcal{D}}\right)$ be the parabolically convex solution to
\[
\left\{\begin{array}{ll}
-u_t\det D_x^2u\le\Lambda & \text { in } \mathcal{D}, \\
m_1\le-u_t\le m_2 & \text { in } \mathcal{D}, \\
u=0 & \text { on } \partial_p\mathcal{D},
\end{array}\right.
\]
where $\Lambda,m_1,m_2$ are positive constants. Then there exists a constant $C>0$ depending only on $n,\Lambda,m_2,\text{diam}\left(\mathcal{D}\right)$, such that for every $X\in \mathcal{D}$
\begin{equation}\label{lower}
u\left(X\right)\ge-C\text{dist}_p\left(X,\partial_p\mathcal{D}\right)^{\frac{1}{n+1}},
\end{equation}
\begin{equation}\label{upper}
u\left(X\right)\le -m_1\text{dist}_p\left(X,\partial_p\mathcal{D}\right)^2.
\end{equation}
\end{lemma}

\begin{proof}
Let $X=(x,t)\in \mathcal{D}$ and $X_1=\left(x_1,t_1\right)\in\partial_p \mathcal{D}$ be such that 
\[
\text{dist}_p\left(X,\partial_p \mathcal{D}\right)=\left|X-X_1\right|_p.
\]
Then for $\epsilon:=-u\left(X\right)$, we have that
\[
u(x,t_1)\le u(x,t)+m_2\cdot\frac{\epsilon}{2m_2}=-\frac{\epsilon}{2}
\]
for $t_1\ge t-\frac{\epsilon}{2m_2}$. We notice that 
\[
|u(x,t_1)|^{n+1}\le C\text{dist}\left(x,\partial \mathcal{D}\left(t_1\right)\right)
\]
by $\mathcal{D}(t)\subset \mathcal{D}\left(t_1\right)$ and Theorem 2.1 in \cite{gh98}. Thus
\[
\left|X-X_1\right|_p\ge|x-x_1|\ge\text{dist}\left(x,\partial \mathcal{D}\left(t_1\right)\right)\ge C\epsilon^{n+1}.
\]
If $t_1\le t-\frac{\epsilon}{2m_2}$,
\[
\left|X-X_1\right|_p\ge|t-t_1|^{\frac{1}{2}}\ge \left(\frac{\epsilon}{2m_2}\right)^{\frac{1}{2}}\ge C\epsilon^{n+1},
\]
where Theorem 2.1 in \cite{gh98} has been used. Until now we have proved \eqref{lower}. 

To prove \eqref{upper}, we notice that for $X\in \mathcal{D}$, there exists $X'=(x,t')\in\partial_p\mathcal{D}$. Then $u\left(X\right)\le -m_1(t-t')\le-m_1\text{dist}_p\left(X,\partial_p \mathcal{D}\right)^2$.
\end{proof}

\begin{remark}
This lemma is slightly different from Proposition 4.1 in \cite{gh01}, where Gutiérrez-Huang established the result for $\inf\limits_{\mathcal{D}}u=0$ and $u=1$ on $\partial_p\mathcal{D}$.
\end{remark}

Next we derive interior gradient and Hessian estimates for smooth solutions.

\begin{lemma}\label{lode}
Let $\mathcal{D}\subset\mathbb{R}^{n+1}_-$ be a bounded bowl-shaped domain and $u\in C^{4,2}\left(\mathcal{D}\right)\cap C^0\left(\overline{\mathcal{D}}\right)$ be the parabolically convex solution to 
\[
\left\{\begin{array}{ll}
-u_t\det D_x^2u=1 & \text { in } \mathcal{D}, \\
m_1\le-u_t\le m_2 & \text { in } \mathcal{D}, \\
u=0 & \text { on } \partial_p \mathcal{D}.
\end{array}\right.
\]
Then there exists a positive constant $C$ depending only on $n,m_1,m_2,\text{diam}\left(\mathcal{D}\right)$, such that, for all $X\in \mathcal{D}$, the following gradient and Hessian estimates hold:
\begin{equation}\label{ge}
|D_xu\left(X\right)|\le C \text{dist}_p\left(X,\partial_p\mathcal{D}\right)^{-2n},
\end{equation}
\begin{equation}\label{he}
|D^2_xu\left(X\right)|\le C \text{dist}_p\left(X,\partial_p\mathcal{D}\right)^{-4n-2}.
\end{equation}
\end{lemma} 

\begin{proof}
For $\delta>0$, define 
\[
\mathcal{D}':=\{X\in \mathcal{D}: u\left(X\right)<-\delta\},\quad w:=u+\delta.
\]
We have that $w\in C^{4,2}\left(\overline{\mathcal{D}'}\right)$ is a parabolically convex solution to the boundary value problem:
\[
\left\{\begin{array}{ll}
-w_t\det D_x^2w=1 & \text { in } \mathcal{D}', \\
w=0 & \text { on } \partial_p \mathcal{D}'.
\end{array}\right.
\]
By \eqref{lower}, there exists $C>0$ such that
\[
\text{dist}_p\left(\mathcal{D}',\partial_p\mathcal{D}\right)\ge C\delta^{n+1}.
\]
Since $w$ is convex in $x$, it follows that for any $X=(x,t)\in \mathcal{D}'$
\begin{equation}\label{first}
|D_xw(x,t)|\le \frac{\delta}{\text{dist}\left(x,\partial \mathcal{D}(t)\right)}\le \frac{\delta}{\text{dist}_p\left(X,\partial _p\mathcal{D}\right)}\le C\delta^{-n}.
\end{equation}

Define the auxiliary function: 
\[
h:=\log\left(-w\right)+\log w_{11}+\frac{\delta^{2n}w_1^2}{2},
\]
the last term of which is bounded by \eqref{first}. Since $w\equiv0$ on $\partial_p\mathcal{D}'$, the maximum point of $h$ is achieved in $\overline{\mathcal{D}'}\setminus\partial_p\mathcal{D}'$. By performing a translation and rotation of coodinates, we may assume that the maximum point of $h$ is $\left(0_n,0\right)$ and $\left(w_{ij}\left(0_n,0\right)\right)$ is diagonal. Then at $\left(0_n,0\right)$, the following conditions hold:
\begin{equation}\label{h_i=0}
h_i=\frac{w_i}{w}+\frac{w_{11i}}{w_{11}}+\delta^{2n}w_1w_{1i}=0,
\end{equation}
\[
h_{ii}=\frac{w_{ii}w-w_i^2}{w^2}+\frac{w_{11ii}w_{11}-w_{11i}^2}{w_{11}^2}+\delta^{2n}w_{1i}^2+\delta^{2n}w_1w_{1ii}\le0,
\]
\[
h_t=\frac{w_t}{w}+\frac{w_{11t}}{w_{11}}+\delta^{2n}w_1w_{1t}\ge0.
\]
At $\left(0_n,0\right)$, we derive the following key inequality:
\[
\begin{aligned}
0&\ge\frac{1}{w_t}h_t+\sum_i\frac{h_{ii}}{w_{ii}}\\
&=\frac{1}{w_t}\left(\frac{w_t}{w}+\frac{w_{11t}}{w_{11}}+\delta^{2n}w_1w_{1t}\right)+\frac{n}{w}-\sum_i\frac{w_i^2}{w^2w_{ii}}+\frac{1}{w_{11}}\sum_i\frac{w_{11ii}}{w_{ii}}\\
&\quad-\sum_i\frac{1}{w_{ii}}\left(\frac{w_{11i}}{w_{11}}\right)^2+\delta^{2n}\sum_i\frac{w_{1i}^2}{w_{ii}}+\delta^{2n}w_1\sum_i\frac{w_{1ii}}{w_{ii}},
\end{aligned}
\]
where $\sum\limits_i\frac{w_{1i}^2}{w_{ii}}=w_{11}$. It follows that at $\left(0_n,0\right)$
\begin{equation}\label{delta2n}
\begin{aligned}
\delta^{2n}w_{11}\left|w\right|&\le n+1+\frac{w}{w_t}\frac{w_{11t}}{w_{11}}+\frac{\delta^{2n}}{w_t}ww_1w_{1t}-\sum_i\frac{w_i^2}{ww_{ii}}+\frac{w}{w_{11}}\sum_i\frac{w_{11ii}}{w_{ii}}\\
&\quad-\sum_i\frac{w}{w_{11}}\left(\frac{w_{11i}}{w_{11}}\right)^2+\delta^{2n}ww_1\sum_i\frac{w_{1ii}}{w_{ii}}.
\end{aligned}
\end{equation}
Let $L:=\frac{1}{w_t}D_t+\sum\limits_{i,j}w^{ij}D_{x_ix_j}$ be the linearized operator of $\log\left(-w_t\det D^2_xw\right)=0$, where we denote $\left(w^{ij}\right)$ as $D_x^2w$'s inverse matrix. We compute in $\mathcal{D}'$
\begin{equation}\label{Lw_1}
Lw_1=\frac{1}{w_t}w_{1t}+\sum_{i,j}w^{ij}w_{1ij}=\left[\log\left(-w_t\det D_x^2w\right)
\right]_1=0,
\end{equation}
\begin{equation}\label{Lw_{11}}
\begin{aligned}
Lw_{11}&=\frac{1}{w_t}w_{11t}+\sum_{i,j}w^{ij}w_{11ij}\\
&=\left[\log\left(-w_t\det D_x^2w\right)
\right]_{11}+\sum_{i,j}\frac{1}{w_t^2}w_{1t}^2+\sum_{i,j}w^{ii}w^{jj}w_{1ij}^2\\
&=\frac{1}{w_t^2}w_{1t}^2+\sum_{i,j}w^{ii}w^{jj}w_{1ij}^2,
\end{aligned}
\end{equation}
where we have used
\[
\left(w^{ij}\right)_1=-\sum_{i,j}\frac{w_{1ij}}{w_{ii}w_{jj}}
\]
in the second equality of \eqref{Lw_{11}}. Then by \eqref{delta2n}, \eqref{Lw_1} and \eqref{Lw_{11}}
\[
\begin{aligned}
\delta^{2n}w_{11}|w|&\le n+1+\frac{w}{w_{11}}Lw_{11}+\delta^{2n}ww_1Lw_1-\sum_{i}\frac{w_i^2}{ww_{ii}}-\sum_i\frac{w}{w_{ii}}\left(\frac{w_{11i}}{w_{11}}\right)^2\\
&=n+1+\frac{w}{w_{11}}\left(\frac{1}{w_t^2}w_{1t}^2+\sum_{i,j}\frac{w_{1ij}^2}{w_{ii}w_{jj}}\right)-\sum_{i}\frac{w_i^2}{ww_{ii}}-\sum_i\frac{w}{w_{ii}}\left(\frac{w_{11i}}{w_{11}}\right)^2.
\end{aligned}
\]
From \eqref{h_i=0} we know that at $\left(0_n,0\right)$ 
\[
\frac{w_{11i}}{w_{11}}=-\frac{w_i}{w}
\]
for $2\le i\le n$. Then
\[
\begin{aligned}
\delta^{2n}w_{11}\left|w\right|&\le n+1+\frac{w}{w_{11}}\sum_{i,j}\frac{w_{1ij}^2}{w_{ii}w_{jj}}-\frac{w_1^2}{ww_{11}}-\sum_{i\ge2}\frac{w_i^2}{ww_{ii}}-\sum_i\frac{w}{w_{ii}}\left(\frac{w_{11i}}{w_{11}}\right)^2\\
&=n+1+\frac{w}{w_{11}}\sum_{i,j}\frac{w_{1ij}^2}{w_{ii}w_{jj}}-\frac{w_1^2}{ww_{11}}- \sum_{i\ge2}\frac{w}{w_{ii}}\left(\frac{w_{11i}}{w_{11}}\right)^2 -\sum_i\frac{w}{w_{ii}}\left(\frac{w_{11i}}{w_{11}}\right)^2\\
&\le n+1+\frac{w}{w_{11}}\sum_{i,j\ge2}\frac{w_{1ij}^2}{w_{ii}w_{jj}}-\frac{w_1^2}{ww_{11}}.
\end{aligned}
\] 
At $\left(0_n,0\right)$, this implies that
\[
\delta^{2n}w_{11}|w|e^{\frac{\delta^{2n}w_1^2}{2}}\le (n+1)e^{\frac{\delta^{2n}w_1^2}{2}}+\frac{w_1^2e^{\delta^{2n}w_1^2}}{\left|w\right|w_{11}e^{\frac{\delta^{2n}w_1^2}{2}}},
\]
\[
\delta^{2n}e^h\le C+\frac{C}{\delta^{2n}e^h},\quad \delta^{2n}e^h\le C.
\]
Consequently for $X\in \mathcal{D}'':=\left\{X\in \mathcal{D}':w<-\delta\right\}$, we have that
\[
|u_{11}\left(X\right)|=\left|w_{11}\left(X\right)\right|\le C\delta^{-2n-1}.
\]
Since the $x_1$ direction is choosen arbitrary, we have that 
\begin{equation}\label{plevel}
\left|D_x^2u\right|\le C\delta^{-2n-1}, 
\end{equation}
where $C=C(n,m_2,r)$. 

For any fixed $X\in \mathcal{D}$, take $\delta:=-\frac{u\left(X\right)}{4}$. By the above interior estimates
\begin{equation}\label{primary}
|D_x^2u\left(Y\right)|\le C\delta^{-2n-1},\quad |D_xu\left(Y\right)|\le C\delta^{-n},\quad Y\in \mathcal{D}''.
\end{equation}
From the estimate \eqref{upper}, 
\[
u\left(X\right)\le -m_1\text{dist}_p\left(X,\partial_p\mathcal{D}\right)^2, \quad X\in \mathcal{D}.
\]
Substituting $\delta=-\frac{u\left(X\right)}{4}\ge \frac{m_1}{4}\text{dist}_p\left(X,\partial_p\mathcal{D}\right)^2$ into \eqref{primary} yields
\[
|D_x^2u\left(Y\right)|\le C\text{dist}_p\left(X,\partial_p\mathcal{D}\right)^{-4n-2},\quad |D_xu\left(Y\right)|\le C\text{dist}_p\left(X,\partial_p\mathcal{D}\right)^{-2n}.
\]
Picking $Y=X$ gives \eqref{ge} and \eqref{he}.

For $n=1$, $\left|D_xu\left(X\right)\right|\le C\delta^{-1}$ for $X\in \mathcal{D}'$ is still established. Then $\left|D_xu\left(X\right)\right|\le C\text{dist}_p\left(X,\partial_p\mathcal{D}\right)^{-2}$.	For Hessian, we have that $\frac{1}{m_2}\le D_x^2u\le \frac{1}{m_1}$.
\end{proof}

\begin{remark}
This lemma provides a ‌refined‌ version of the parabolic Pogorelov estimate established in Theorem 2.2 of \cite{gh98} by Gutiérrez-Huang. Moreover, it serves as the ‌parabolic counterpart‌ to Lemma 1.1 in \cite{cl04}, extending Caffarelli-Li's elliptic result to the parabolic setting.
\end{remark}

Now we want to establish interior estimates for higher-order derivatives, which have a power-type dependence on the parabolic distance to boundary. In the following Lemma \ref{uhode} we will accomplish it by normalizing level set. 

\begin{lemma}\label{uhode}
Under the conditions in Lemma \ref{lode}, let $u\in C^\infty\left(\mathcal{D}\right)$. Then for every $k>2$, there exist positive constants $\beta_{k},C_{k}$ depending only on $n,k,m_1,m_2,\text{diam}\left(\mathcal{D}\right)$, such that for all $i,j\in\mathbb{N}$ and $i+2j=k$, the following estimate holds:
\begin{equation}\label{hode}
\left|D_x^iD_t^ju\left(X\right)\right|\le C_{k}\text{dist}_p\left(X,\partial_p \mathcal{D}\right)^{-\beta_{k}}.
\end{equation}
\end{lemma}

\begin{proof}
For $(x,t)\in \mathcal{D}$, define 
\[
d:=\frac{1}{2}\text{dist}_p\left((x,t),\partial_p\mathcal{D}\right)
\]
and consider the function
\begin{equation}\label{defofv}
v(y,s):=u(x+y,t+s)-u(x,t)-D_xu(x,t)\cdot y.
\end{equation}
We first consider the case $n\ge2$. By Lemma \ref{lode}, $-v_s\det D_y^2v=1$ and $-v_s\le m_2$, we obtain the Hessian estimate
\begin{equation}\label{ec}
C^{-1}d^{(n-1)(4n+2)}\le D_y^2v(y,s)\le Cd^{-4n-2}
\end{equation}
for $(y,s)\in E_d$. 

\textbf{Step 1: Choose an appropriate level set of $v$.} We claim that there exist constants $C_0,C_1>0$ depending only on $n,m_1,m_2,\text{diam}\left(\mathcal{D}\right)$, such that
\begin{equation}\label{BHB}
B_{C_1d^{2n^2+n+1}}\subset\left\{y\in \mathbb{R}^{n}: v(y,0)<H \right\}\subset B_{\frac{d}{2}},
\end{equation}
where $H=C_0d^{4n^2-2n}$. Actually for $y\in B_1$, the Hessian estimate \eqref{ec} implies
\[
v(y,0)\ge C^{-1}d^{(n-1)(4n+2)}|y|^2.
\]
Choosing $C_0$ sufficiently small, we obtain
\[
\left\{y\in \mathbb{R}^{n}: v(y,0)<H \right\}\subset\left\{ y\in \mathbb{R}^{n}: C^{-1}d^{(n-1)(4n+2)}|y|^2<H \right\}\subset B_{\frac{d}{2}}.
\] 
Similarly, \eqref{ec} gives
\[
v(y,0)\le Cd^{-4n-2}|y|^2
\]
for $y\in B_1$. For $C_1$ small enough,
\[
B_{C_1d^{2n^2+n+1}}\subset \left\{ y\in \mathbb{R}^{n}: Cd^{-4n-2}|y|^2<H \right\} \subset \left\{y\in \mathbb{R}^{n}: v(y,0)<H \right\}. 
\]
This completes the proof of \eqref{BHB}.

Define the sets 
\[
\mathcal{Q}_H:=\{Y\in \mathbb{R}^{n+1}_-:v\left(Y\right)<H\}
\] 
and 
\[
\mathcal{Q}_H(0):=\{y\in \mathbb{R}^{n}:v\left(y,0\right)<H\}.
\] 

By John's Lemma (see Theorem 1.8.2 in \cite{g16}), there exists an affine transformation $A_Hy:=a_Hy+b_H$ with $\det a_H=1$ and $b_H\in \mathbb{R}^n$, such that 
\[
B_{\alpha_nR}\subset A_H\left(\mathcal{Q}_H(0)\right)\subset B_{R} 
\]
for some $R>0$ depending on $H$, where $\alpha_n:=n^{-\frac{3}{2}}$. Combining this with \eqref{BHB}, we derive the estimate
\begin{equation}\label{eod}
C^{-1}d^{2n^2+n+1}\le R\le Cd,
\end{equation} 
which will be used later.

For sufficiently small $C_0$, the condition $-v_s\ge m_1$ ensures $\mathcal{Q}_H\subset\subset E_d$.

\textbf{Step 2: Normalize $v$ and $\mathcal{Q}_H$.} Following the proof of Lemma 3.1 in \cite{gh98}, there exist $\epsilon_0,\epsilon_1,\epsilon_2>0$, such that 
\[
B_{\epsilon_0R}\times\left[-\epsilon_1H,0\right]\subset \begin{pmatrix}A_H &0_n\\0_n'&1\end{pmatrix}\mathcal{Q}_H\subset B_{nR}\times \left[-\epsilon_2H,0\right].
\]
We define the normalization mapping
\[
\gamma_H(y,s):=\left(\frac{1}{R}A_Hy,\frac{s}{R^2}\right),\quad (y,s)\in \mathcal{Q}_H, 
\]
and the normalized function
\[
v^*(\textbf{y},\textbf{s}):=\frac{1}{R^2}\left(v\circ\gamma_H^{-1}(\textbf{y},\textbf{s})-H\right),\quad (\textbf{y},\textbf{s})\in \mathcal{Q}_H^*:=\gamma_H\mathcal{Q}_H.
\]
It follows that
\[
\frac{1}{m_2}\le \det D_{\textbf{y}}^2v^*\le \frac{1}{m_1}\quad \text{in}\quad \mathcal{Q}_H^*(0).
\]
Due to Proposition 1.1 in \cite{gh00}, there exist $C_1,C_2>0$ such that
\[
C_1\left|\min_{\mathcal{Q}_H^*(0)}v^*(\cdot,0)\right|^n\le\int_{\mathcal{Q}_H^*(0)}\det D_\textbf{y}^2v^*(\textbf{y},0)d\textbf{y}\le C_2\left|\min_{\mathcal{Q}_H^*(0)}v^*(\cdot,0)\right|^n.
\]
This implies that $C^{-1}\le HR^{-2}\le C$ for $C>1$. The normalized function $v^*$ satisfies the following boundary value problem: 
\[
\left\{\begin{array}{ll}
-v^*_\textbf{s}\det D_\textbf{y}^2v^*=1 & \text { in } \mathcal{Q}_H^*, \\
m_1\le-v^*_\textbf{s}\le m_2 & \text { in } \mathcal{Q}_H^*, \\
v^*=0 & \text { on } \partial_p \mathcal{Q}_H^*.
\end{array}\right.
\]

\textbf{Step 3: Estimate $v^*$'s higher order derivatives.} Observe that $v^*\left(\frac{b_H}{R},0\right)=-HR^{-2}$. Combined with \eqref{upper}, we obtain the $H$-uniform lower bound 
\[
\text{dist}_p\left(\left(\frac{b_H}{R},0\right),\partial_p\mathcal{Q}_H^*\right)\ge C^{-1} >0.
\]
By applying Lemma \ref{lode}, the interior Evans-Krylov estimates and Schauder estimates, we derive the following $H$-uniform bounds on higher-order derivatives:
\begin{equation}\label{constant}
\left|D_\textbf{y}^iD_\textbf{s}^jv^*\left(\frac{b_H}{R},0\right)\right|\le C\left(k,n,m_1,m_2,\text{diam}\left(\mathcal{D}\right)\right) \quad \text{for}\quad i,j\in\mathbb{N},i+2j= k.
\end{equation}

\textbf{Step 4: Return to estimates of $v$.} The normalized solution satisfies the uniform bound
\[
D_\textbf{y}^2v^*\left(\frac{b_H}{R},0\right)\le CI.
\]
From the Hessian estimate \eqref{ec}, we derive the matrix inequality
\[
\left(a_H^{-1}\right)'D_y^2v\left(0_n,0\right)a_H^{-1}\ge C^{-1}d^{(4n+2)(n-1)}\left(a_H^{-1}\right)'a_H^{-1},
\]
which implies the bound 
\[
\|a_H^{-1}\|\le Cd^{(2n+1)(1-n)}. 
\]
Since $\det a_H=1$, the reverse bound follows:
\[
\|a_H\|\le Cd^{-(2n+1)(n-1)^2}.
\] 
Combining this with the radius estimate \eqref{eod}, we obtain the derivative control
\[
\left|D_x^iD_t^ju\left(X\right)\right|=\left|D_y^iD_s^jv\left(0_n,0\right)\right|\le R^{2-i-2j}\left\|a_H\right\|^i\left|D_\textbf{y}^iD_\textbf{s}^jv^*\left(\frac{b_H}{R},0\right)\right|\le Cd^{-\beta_k},
\] 
where the exponent is 
\[
\beta_k:=\left(k-2\right)\left(2n^2+n+1\right)+k\left(2n+1\right)\left(n-1\right)^2,
\] 
for $i,j\in\mathbb{N}$ with $i+2j=k$.

The estimate for $n=1$ can be proved similarly from that $\frac{1}{m_2}\le D_{x}^2u\le \frac{1}{m_1}$.
\end{proof}

\quad

\section{A homogenization estimate}\label{2}
Let $\{\epsilon_1,\cdots,\epsilon_n\}$ be a basis of $\mathbb{R}^n$ and $\epsilon_0>0$ a fixed constant. Let $f_1\in C^{\alpha}\left(\mathbb{R}^n\right)$ be a positive function satisfying that
\[
-\mkern-19mu\int_{\left\{x\in\mathbb{R}^n:x=\sum\limits_{i=1}^n\lambda_i\epsilon_i,0\le\lambda_i\le1,0\le i\le n\right\}}f_1(x)dx=1
\] 
and 
\[
f_1\left(x+\epsilon_i\right)=f_1(x) 
\]
for any $x\in\mathbb{R}^n$ and $1\le i\le n$ and $f_2\in C^{\frac{\alpha}{2}}\left(\mathbb{R}^1_-\right)$ be a positive function satisfying that 
\[
-\mkern-19mu\int_{\left[-\epsilon_0,0\right]}f_2(t)dt=1,\quad f_2\left(t-\epsilon_0\right)=f_2(t)
\]
for every $t\in\mathbb{R}^1$. Let $f:=f_1f_2$, $\mathcal{D}\subset\mathbb{R}^{n+1}_-$ be a bounded bowl-shaped domain, 
and $\overline{w}\in C^\infty\left(\mathcal{D}\right)\cap C^0\left(\overline{\mathcal{D}}\right)$
 be the parabolically convex solution of
\begin{equation}\label{ww92}
\left\{\begin{array}{ll}
-\overline{w}_t\det D_x^2\overline{w}=1 & \text { in } \mathcal{D}, \\
m_1\le-\overline{w}_t\le m_2 & \text { in } \mathcal{D}, \\
\overline{w}=0 & \text { on } \partial_p \mathcal{D}.
\end{array}\right.
\end{equation}
In order to establish the homogenization estimate, we construct a periodic corrector function that captures the oscillatory nature of $f$.

\begin{lemma}\label{exist}
There exists a constant $C(n)>0$ such that
\[
\|\xi\|_{C^0\left(\mathbb{R}^{n+1}\right)}\le C(n)\left(\|D_x^2\overline{w}\left(Z\right)\|\sum_{i=1}^n|\epsilon_i|^2-\overline{w}_t\left(Z\right)\epsilon_0\right).
\]
\end{lemma}

\begin{proof}
By Theorem 0.1 in \cite{cl04}, there exists $\xi_1\in C^2\left(\mathbb{R}^n\right)$ satisfying the equation
\[
\det\left(D_x^2\overline{w}\left(Z\right)+D^2_x\xi_1\right)=\det D_x^2\overline{w}\left(Z\right)f_1 \quad \text{in} \quad \mathbb{R}^n,
\] 
the periodicity condition
\[
\xi_1(x+\epsilon_i)=\xi_1(x)\quad \text{for every}\quad x\in\mathbb{R}^n \quad \text{and}\quad 1\le i\le n 
\]
and the estimate
\begin{equation}\label{xi1}
\|\xi_1\|_{C^0\left(\mathbb{R}^n\right)}\le C(n)\|D_x^2\overline{w}\left(Z\right)\|\sum_i |\epsilon_i |^2.
\end{equation}

Now we define the temporal corrector $\xi_2$ and give its estimate. Take 
\[
\xi_2(t):=\overline{w}_t\left(Z\right)\int_0^t\left(f_2(s)-1\right)ds,
\]
which inherits the periodicity
\[
\xi_2(t-\epsilon_0)=\xi_2(t),
\]
and solves
\[
-\left(\overline{w}_t\left(Z\right)+\xi_{2,t}\right)=-\overline{w}_t\left(Z\right)f_2 \quad \text{in}\quad \mathbb{R}_-^1.
\]
To estimate $\|\xi_2\|_{C^0\left(\mathbb{R}_-^1\right)}$, we observe that $\xi_2(0)=\xi_2\left(-\epsilon_0\right)=0$. For $t\in\left[-\epsilon_0,0\right]$, the fundamental theorem of calculus yields
\[
\xi_2\left(t\right)-\xi_2\left(-\epsilon_0\right)=\int_{-\epsilon_0}^t \xi_2'(s)ds \le -\overline{w}_t\left(Z\right)\epsilon_0,
\]
\[
\xi_2\left(0\right)-\xi_2\left(t\right)=\int_t^0 \xi_2'(s)ds \le -\overline{w}_t\left(Z\right)\epsilon_0.
\]
Thus
\begin{equation}\label{xi2}
\|\xi_2\|_{C^0\left(\mathbb{R}_-^1\right)}\le -\overline{w}_t\left(Z\right)\epsilon_0.
\end{equation}

Denote $\xi:=\xi_1+\xi_2$, then $\xi\in C^{2,1}\left(\mathbb{R}^{n+1}_-\right)$ satisfies the equation
\[
-\left(\overline{w}_t\left(Z\right)+\xi_t\right)\det\left(D_x^2\overline{w}\left(Z\right)+D_x^2\xi\right)=f\quad \text{in}\quad \mathbb{R}^{n+1}_-.
\]
The desired bound follows directly from \eqref{xi1} and \eqref{xi2}.
\end{proof}

Let $w\in C^{2,1}\left(\mathcal{D}\right)\cap C^0\left(\overline{\mathcal{D}}\right)$ be the parabolically convex solution of
\begin{equation}\label{wequation}
\left\{\begin{array}{ll}
-w_t\det D_x^2w=f & \text { in } \mathcal{D}, \\
m_1\le-w_t\le m_2 & \text { in } \mathcal{D}, \\
w=0 & \text { on } \partial_p \mathcal{D},
\end{array}\right.
\end{equation}
Employing the periodic corrector, we obtain the homogenization estimate.

\begin{lemma}\label{homogenization}
Let $\mathcal{D}\subset E_r$ for some $r>0$, $w$ satisfies \eqref{wequation} and $\overline{w}$ satisfies \eqref{ww92}. Then there exist constants $\beta,C>0$, depending only on $n,\Lambda,m_1,m_2,r$, such that the following estimate holds:
\[
\|w-\overline{w}\|_{C^0\left(\overline{\mathcal{D}}\right)}\le C\left(\sum_i|\epsilon_i|^2+\epsilon_0\right)^{\beta}.
\]
\end{lemma}

\begin{proof}

Let $M:=\sup\limits_{\mathcal{D}}|w-\overline{w}|$. We only treat the case $M=\sup\limits_{\mathcal{D}}\left(w-\overline{w}\right)>0$ since the case $M=\sup\limits_{\mathcal{D}}\left(\overline{w}-w\right)$ is analogous. Let $\overline{X}=\left(\overline{x},\overline{t}\right)$ be a maximum point of $w-\overline{w}$. By the lower bound \eqref{lower}, we have that
\begin{equation}\label{Mbdd}
M=w\left(\overline{X}\right)-\overline{w}\left(\overline{X}\right)\le-\overline{w}\left(\overline{X}\right)\le C\text{dist}_p\left(\overline{X},\partial_p\mathcal{D}\right)^{\frac{1}{n+1}}\le C\left(n,m_2,r\right).
\end{equation}
This implies
\[
\text{dist}_p\left(\overline{X},\partial_p\mathcal{D}\right)\ge\frac{1}{C}M^{n+1}=:\delta,
\]
which tells that $\overline{X}$ is far from the boundary.

Let 
\[
u(x,t):=w(x,t)+\frac{M}{36r^2}\left|x-\overline{x}\right|^2-\frac{M}{9r^2}\left(t-\overline{t}\right).
\]
Then $\left(u-\overline{w}\right)\left(\overline{X}\right)=M$. Since $\mathcal{D}\subset E_r$
\[
\left|u-w\right|\le \frac{2M}{9} \quad\text{in}\quad \mathcal{D}.
\]
By the boundary condition $w=\overline{w}$ on $\partial_p\mathcal{D}$, it follows that $\left|u-\overline{w}\right|\le \frac{2M}{9}$ on $\partial_p\mathcal{D}$. Thus there exists $\tilde{X}=\left(\tilde{x},\tilde{t}\right)\in \mathcal{D}$ such that
\[
\left(u-\overline{w}\right)\big(\tilde{X}\big)=\sup\limits_{\mathcal{D}}\left(u-\overline{w}\right)\ge M.
\] 
Additionally we have that
\[
\left(w-\overline{w}\right)\big(\tilde{X}\big)=\left[\left(u-\overline{w}\right)-\left(u-w\right)\right]\big(\tilde{X}\big)\ge M-\frac{2M}{9}=\frac{7M}{9}.
\]
Combined with $\left(w-\overline{w}\right)\le -\overline{w}$ and \eqref{lower}, we derive that
\[
\text{dist}_p\left(\tilde{X},\partial_p\mathcal{D}\right)\ge\frac{1}{C}M^{n+1}=:\delta.
\]
By Lemma \ref{uhode}, for $X\in E_{\frac{\delta}{2}}\big(\tilde{X}\big)$, the higher order derivatives of $\overline{w}$ satisfy
\begin{equation}\label{lagrange}
\begin{aligned}
&\left|D_x^3\overline{w}\left(X\right)\right|+\left|D_x^2\overline{w}_t\left(X\right)\right|+\left|D_x\overline{w}_t\left(X\right)\right|+\left|\overline{w}_{tt}\left(X\right)\right|\\
&\le C\left( \text{dist}_p\left(X,\partial_p\mathcal{D}\right)^{-\beta_3}+\text{dist}_p\left(X,\partial_p\mathcal{D}\right)^{-\beta_4} \right)\\
&\le C\text{dist}_p\left(X,\partial_p\mathcal{D}\right)^{-\beta_4}\le CM^{-\beta_4(n+1)}.
\end{aligned}
\end{equation}

Define
\[
v\left(X\right):=\overline{w}\left(X\right)+\xi\left(X\right)-\frac{M}{36r^2}\left|x-\overline{x}\right|^2+\frac{M}{9r^2}\left(t-\overline{t}\right)+\frac{M}{144r^2}\left|x-\tilde{x}\right|^2-\frac{M}{18r^2}\left(t-\tilde{t}\right),
\]
where $\xi$ is the periodic corrector defined in Lemma \ref{exist} for $Z=\tilde{X}$. Then
\[
\left(w-v\right)\left(X\right)=u\left(X\right)-\left( \overline{w}\left(X\right) + \xi\left(X\right)   + \frac{M}{144r^2}\left|x-\tilde{x}\right|^2 -\frac{M}{18r^2}\left(t-\tilde{t}\right) \right).
\]

Take $\delta_1:=\frac{M^{\beta_4(n+1)+1}}{C_0}$, where $C_0>0$ is a sufficiently large constant to be determined. Owing to \eqref{Mbdd}, we may choose $C_0$ such that 
\[
\delta_1=\frac{M^{n+1}}{C}\cdot \frac{CM^{(\beta_4-1)(n+1)+1}}{C_0}<\frac{\delta}{2}. 
\]
For $X\in E_{\delta_1}\big(\tilde{X}\big)$, the derivative estimates of $\left|D_x^3\overline{w}\right|,\left|D_x^2\overline{w}_t\right|$ in \eqref{lagrange} yield
\begin{equation}\label{analysis}
\begin{aligned}
D_x^2v\left(X\right)&=D_x^2\overline{w}\left(X\right)+D_x^2\xi\left(X\right)-\frac{M}{24r^2}I\\
&\le D_x^2\overline{w}\big(\tilde{X}\big)+CM^{-\beta_4(n+1)}\left(\left|x-\tilde{x}\right|+\left|t-\tilde{t}\right|\right)I+D_x^2\xi\left(X\right)-\frac{M}{24r^2}I\\
&\le D_x^2\overline{w}\big(\tilde{X}\big)+CM^{-\beta_4(n+1)}\left|X-\tilde{X}\right|_pI+D_x^2\xi\left(X\right)-\frac{M}{24r^2}.
\end{aligned}
\end{equation}
In the last line of \eqref{analysis}, by further increasing $C_0$, we ensure that
\[
CM^{-\beta_4(n+1)}\left|X-\tilde{X}\right|_p<CM^{-\beta_4(n+1)}\frac{M^{\beta_4(n+1)+1}}{C_0}<\frac{M}{24r^2},
\]
which implies that for $X\in E_{\delta_1}\big(\tilde{X}\big)$
\[
D_x^2v\left(X\right)<D_x^2\overline{w}\big(\tilde{X}\big)+D_x^2\xi\left(X\right).
\]
Similar to \eqref{analysis}, by estimates of $\left|D_x\overline{w}_t\right|,\left|\overline{w}_{tt}\right|$ in \eqref{lagrange}, we obtain that for $C_0$ large
\[
v_t\left(X\right)\ge \overline{w}_t\big(\tilde{X}\big)-CM^{-\beta_4(n+1)}\left|X-\tilde{X}\right|_p+\xi_t\left(X\right)+\frac{M}{18r^2}>\overline{w}_t\big(\tilde{X}\big)+\xi_t\left(X\right).
\]
Then for $X\in E_{\delta_1}\big(\tilde{X}\big)$ satisfying $D_x^2v\left(X\right)>0$ and $v_t\left(X\right)<0$, we derive the inequality
\begin{equation}\label{contradiction}
\begin{aligned}
\left(-v_t\det D_x^2v\right)\left(X\right)&<-\left(\overline{w}_t\big(\tilde{X}\big)+\xi_t\left(X\right)\right)\det\left(D_x^2\overline{w}\big(\tilde{X}\big)+D_x^2\xi\left(X\right)\right)\\
&=f(x,t)=\left(-w_t\det D_x^2w\right)\left(X\right).
\end{aligned}
\end{equation}

At $\tilde{X}$, the maximum point of $u-\overline{w}$, there is
\[
\begin{aligned}
w-v=\left(u-\overline{w}\right)-\xi
&\ge\left(u-\overline{w}\right)-C\left(\sum_{i=1}^n|\epsilon_i|^2+\epsilon_0\right)\left(M^{-(4n+2)(n+1)}+1\right)\\
&\ge\left(u-\overline{w}\right)-C\left(\sum_{i=1}^n|\epsilon_i|^2+\epsilon_0\right)M^{-(4n+2)(n+1)}.
\end{aligned}
\]
For $X\in \partial_pE_{\delta_1}\big(\tilde{X}\big)$,
\[
\begin{aligned}
\left(w-v\right)\left(X\right)
&=\left(u-\overline{w}\right)\left(X\right)-\xi\left(X\right)-\frac{M}{144r^2}\left|x-\tilde{x}\right|^2+\frac{M}{18r^2}\left(t-\tilde{t}\right)\\
&\le \left(u-\overline{w}\right)\big(\tilde{X}\big) + C\left(\sum_{i=1}^n|\epsilon_i|^2+\epsilon_0\right)M^{-(4n+2)(n+1)}-\frac{M\delta_1}{C}.
\end{aligned}
\]

If 
\begin{equation}\label{contradict}
2C\left(\sum_{i=1}^n|\epsilon_i|^2+\epsilon_0\right)M^{-(4n+2)(n+1)}< \frac{M^{\beta_4(n+1)+2}}{CC_0},
\end{equation}
then 
$\left(w-v\right)\left(X\right)<\left(w-v\right)\big(\tilde{X}\big)$ for any $X\in \partial_pE_{\delta_1}\big(\tilde{X}\big)$. Let $X_1\in E_{\delta_1}\big(\tilde{X}\big)$ be an interior maximum point of $w-v$, then $0<D_x^2w\le D_x^2v$, $0>w_t\ge v_t$ and $-v_t\det D_x^2v\ge -w_t\det D_x^2w$ at $X_1$. This contradicts  \eqref{contradiction},which implies the converse inequality of \eqref{contradict}. Thus we obtain the desired bound
\begin{equation}\label{result}
M\le C\left(\sum_{i=1}^n|\epsilon_i|^2+\epsilon_0\right)^{\frac{1}{\left(\beta_4+4n+2\right)(n+1)+2}}.
\end{equation}
\end{proof}

\quad

\section{Asymptotic behavior at infinity}\label{3}
In this section we use the nonlinear perturbation method to establish the asymptotic behavior of $u$ at infinity. By appropriate normalization, we may assume that
\[
u\left(0_n,0\right)=0,\quad D_xu\left(0_n,0\right)=0_n.
\]
We name
\[
Q_{H}:=\left\{(x, t) \in \mathbb{R}_{-}^{n+1}: u(x, t)<H\right\},\quad Q_{H}(t):=\left\{x \in \mathbb{R}^{n}: u(x, t)<H\right\}
\]
for $H>0$ and $t\le0$. From Proposition 2.8 in \cite{cl03}, $Q_H(0)$ is guaranteed to be bounded. Let $E$ be the ellipsoid of minimum volume containing $Q_{H}(0)$, centered at the center of mass of $Q_{H}(0)$. Applying John's Lemma (see Theorem 1.8.2 in \cite{g16}), we obtain the containment:
\[
\alpha_nE\subset Q_{H}(0)\subset E,
\]
where $\alpha_n:=n^{-\frac{3}{2}}$ and $\alpha E$ denotes the $\alpha$-dilation of $E$ with respect to its center. Then there exist an affine transformation 
\[
A_Hx:=a_{H} x+b_{H}
\]
with $\operatorname{det} a_{H}=1$, $b_H\in \mathbb{R}^n$ and $R=R(H)>0$ such that
\begin{equation}\label{john}
A_H(E)=B_R,\quad B_{\alpha_nR}\subset A_H\left(Q_{H}(0)\right)\subset B_R.
\end{equation}

We first present a normalization result.

\begin{lemma}\label{normalize}
There exists constants $\epsilon_0,\epsilon_1,\epsilon_2>0$ depending on $n,m_1,m_2$ such that for every $H>0$, the transformed domain satisfies
\begin{equation}\label{normalization1}
B_{\epsilon_0R}\times\left(-\epsilon_1H,0\right]\subset \begin{pmatrix}A_H &0_n\\0_n'&1\end{pmatrix}Q_{H} \subset B_R\times\left(-\epsilon_2H,0\right].
\end{equation}
\end{lemma}

\begin{proof}
From \eqref{u_t}, we have for $t\le0$ that 
\[
u(x,t)=u(x,0)+\int_{0}^{t}u_t(x,s)ds\ge u(x,0) -m_1t\ge -m_1t.
\]
Take $t<-\frac{H}{m_1}$, then $u(x,t)\ge H$, implying
\[
Q_H\subset Q_H(0)\times\left[-\frac{H}{m_1},0\right].
\] 
Combined with the spatial containment \eqref{john}, this yields
\begin{equation}\label{right1}
Q_H\subset B_{R}\times\left[-\frac{H}{m_1},0\right].
\end{equation}
By Lemma 2.1 in \cite{gh00}, the spatial domain satisfies that
\[
\frac{1}{2} Q_H(0) \subset Q_{\left(1-\frac{\alpha_n}{4}\right)H}(0).
\] 
It follows from \eqref{john} that
\[
B_{\frac{\alpha_nR}{2}}\subset Q_{\left(1-\frac{\alpha_n}{4}\right)H}(0).
\]
For $x\in Q_{\left(1-\frac{\alpha_n}{4}\right)H}(0)$ and $t>-\frac{1}{4m_2}\alpha_n H$, \eqref{u_t} gives that
\[
u(x,t)=u(x,0)+\int_{0}^{t}u_t(x,s)ds\le \left(1-\frac{\alpha_n}{4}\right)H-tm_2\le H.
\]
Thus
\begin{equation}\label{left1}
B_{\frac{\alpha_nR}{2}}\times\left[-\frac{\alpha_nH}{4m_2},0\right]\subset Q_H.
\end{equation}
Combined \eqref{right1} and \eqref{left1}, we obtain \eqref{normalization1} for constants
\[
\epsilon_0:=\frac{\alpha_n}{2},\quad \epsilon_1:=\frac{\alpha_n}{4m_2},\quad \epsilon_2:=\frac{1}{m_1}.
\]
\end{proof}

We prove the estimate for $HR^{-2}$, which will be used frequently in this section.

\begin{lemma}\label{hr}
There exists a constant $C>1$ depending only on $n,m_1,m_2,\lambda,\Lambda$, such that for every $H>0$
\[
C^{-1}\le HR^{-2}\le C.
\] 
\end{lemma}

\begin{proof}
For the first inequality, we define the domain
\[
D_{1}:=\left\{(y, s) \in \mathbb{R}_{-}^{n+1}: s>\frac{\varepsilon_{1} H}{\varepsilon_{0}^{2}R^{2}}\left(|y|^{2}-\varepsilon_{0}^{2} R^{2}\right)\right\},
\]
where $\epsilon_0,\epsilon_1$ are from \eqref{normalization1}. Note that
\[
D_{1}\subset B_{\varepsilon_{0} R} \times\left[-\varepsilon_{1} H, 0\right]
\]
and then $u\le H$ on $\partial_pD_1$. Let
\[
U(y,s):=u\left(A_H^{-1}y,s\right),\quad (y,s)\in\begin{pmatrix}A_H &0_n\\0_n'&1\end{pmatrix}Q_{H}
\]
and
\[
w_1(y,s):=\frac{ \lambda^{\frac{1}{n+1}} \epsilon_0^{\frac{2n}{n+1}} R^{\frac{2n}{n+1}} }{ 2^{\frac{2n}{n+1}} \epsilon_1^{\frac{n}{n+1}} H^{\frac{n}{n+1}} } \left( -s+\frac{\epsilon_1H}{\epsilon_0^2R^2} \left( |y|^2-\epsilon_0^2R^2 \right) \right) +H, \quad (y,s)\in D_1.
\]
Then
\[
\left\{\begin{array}{ll}
-w_{1,s} \operatorname{det} D_y^{2} w_1=\lambda\le f &\text { in } D_1, \\
w_1=H>U & \text { on } \partial_{p} D_1,
\end{array}\right.
\]
The comparison principle (see Lemma 2.5 in \cite{zgb21}) yields 
\[
w_1\left(0_n,0\right)\ge U\left(0_n,0\right)=u\left(A_H^{-1}0_n,0\right)\ge0 , 
\]
which gives that $H\ge C^{-1}R^2$ for some constant $C>1$. 

For the second inequality, we define
\[
D_{2}:=\left\{(y, s) \in \mathbb{R}_{-}^{n+1}: s>\frac{\varepsilon_{2} H}{R^{2}}\left(|y|^{2}-2 R^{2}\right)\right\},
\]
where $\epsilon_2$ is the constant in \eqref{normalization}. Observe that
\[
B_R\times\left(-\epsilon_2H,0\right]\subset D_2.
\]
Construct another function
\[
w_2(y,s):=\frac{ \Lambda^{\frac{1}{n+1}} R^{\frac{2n}{n+1}} }{ 2^{\frac{2n}{n+1}} \epsilon_2^{\frac{n}{n+1}} H^{\frac{n}{n+1}} } \left(-s+\frac{\varepsilon_{2} H}{R^{2}}\left(|y|^{2}-2 R^{2}\right)\right)+H,\quad (y,s)\in D_2,
\]
which solves
\[
\left\{\begin{array}{ll}
-w_{2,s} \operatorname{det} D_y^{2} w_2=\Lambda &\text { in } D_2, \\
w_2=H & \text { on } \partial_{p} D_2,
\end{array}\right.
\]
By the comparison principle 
\[
w_2\left(0_n,0\right)\le w_2\left(A_H0_n,0\right)\le U\left(A_H0_n,0\right)=u\left(0_n,0\right)=0,
\]
implying $H\le CR^2$ for a constant $C>1$. The lemma has been proved.
\end{proof}

Combining Lemma \ref{normalize} and \ref{hr}, we obtain the normalized containment
\[
B_{\epsilon_0}\times\left(-C^{-1}\epsilon_1,0\right]\subset \begin{pmatrix}R^{-1}A_H &0_n\\0_n'&R^{-2}\end{pmatrix}Q_{H} \subset B_1\times\left(-C\epsilon_2,0\right].
\]

Define
\[
\Gamma_H(x,t):=\begin{pmatrix}R^{-1}a_H &0_n\\0_n'&R^{-2}\end{pmatrix}\left(x,t\right),\quad Q_H^{*}:=\Gamma_H(Q_H).
\]
For convenience of analysis in $Q_H^*$, we prove the following estimate for $a_H$.

\begin{lemma}\label{normalization}
There exist constants $\epsilon_0',\epsilon_1',\epsilon_2'>0$ depending only on $n,m_1,m_2,\lambda,\Lambda$, such that for every $H>0$
\begin{equation}\label{parabolic case level set}
B_{\epsilon_0'}\times\left(-\epsilon_1',0\right]\subset Q_H^{*}\subset B_2\times\left(-\epsilon_2',0\right].
\end{equation}
\end{lemma}

\begin{proof}
\textbf{Step 1: Spatial Normalization. }We claim that there exists a constant $C>1$ depending only on $n,m_1,m_2,\lambda,\Lambda$ such that for every $H>0$
\begin{equation}\label{elliptic case level set}
B_{\frac{R}{C}} \subset a_{H}\left(Q_{H}(0)\right) \subset B_{2 R}.
\end{equation}
The right containment of \eqref{elliptic case level set} follows directly from $A_H\left(Q_{H}(0)\right)\subset B_{R}$ and $0\in a_{H}(Q_{H}(0))$. For the left containment we define
\[
O_{H}:=\frac{1}{R} A_H\left(Q_{H}(0)\right),
\]
and
\[
W(y)=\frac{1}{R^{2}}\left(\frac{m_{1}}{\Lambda}\right)^{\frac{1}{n}}\left(u\left(A_H^{-1}(Ry), 0\right)-H\right),\quad y\in O_{H}.
\]
$W$ satisfies 
\[
\left\{\begin{array}{ll}
\operatorname{det} D_y^{2} W\le 1 &\text { in } O_{H}, \\
W\le0 & \text { on } \partial O_{H},
\end{array}\right.
\]
with $B_{\alpha_{n} } \subset O_{H} \subset B_{1}$. By the Alexandrov maximum principle (Theorem 2.8 in \cite{f17}) 
\begin{equation}\label{alex}
W(y) \geq-C(n) \operatorname{dist}\left(y, \partial O_{H}\right)^{\frac{1}{n}}, \quad y \in O_{H} .
\end{equation}
For every $y\in \frac{1}{R}A_H\left(Q_{\frac{H}{2}}(0)\right)$, we have that
\begin{equation}\label{wawayfrom0}
W(y)\le\frac{1}{R^{2}}\left(\frac{m_{1}}{\Lambda}\right)^{\frac{1}{n}}\left(\frac{H}{2}-H\right)=-\frac{H}{2R^{2}}\left(\frac{m_{1}}{\Lambda}\right)^{\frac{1}{n}}.
\end{equation}
Combining \eqref{alex}, \eqref{wawayfrom0} and Lemma \ref{hr} yields
\[
\operatorname{dist}\left(y, \partial O_{H}\right)\ge \frac{H^{\frac{n}{2}}}{CR^n}\ge C^{-1}.
\]
Thus for every $z\in A_H(Q_{\frac{H}{2}}(0))$
\[
\operatorname{dist}\left(z, \partial \left\{A_H\left(Q_{H}(0)\right)\right\}\right)\ge C^{-1}R.
\]
Since $0\in a_{H}(Q_{\frac{H}{2}}(0))$, we conclude that $B_{\frac{R}{C}}\subset a_{H}\left(Q_{H}(0)\right)$. The claim has been proved.

\textbf{Step 2: Spatiotemporal Containment.} From \eqref{u_t},  for $t\le0$
\[
u(x,t)\ge u(x,0) -m_1t\ge -m_1t.
\]
Thus for $t<-\frac{H}{m_1}$, we have $u(x,t)\ge H$. This combined with \eqref{elliptic case level set} implies that
\[
Q_H\subset Q_H(0)\times\left[-\frac{H}{m_1},0\right]\subset B_{2R}\times\left[-\frac{H}{m_1},0\right].
\]
Conversely, Lemma 2.1 in \cite{gh00} gives that  
\[
\frac{1}{2} Q_H(0) \subset Q_{\left(1-\frac{\alpha_n}{4}\right)H}(0).
\] 
For every $x\in Q_{\left(1-\frac{\alpha_n}{4}\right)H}(0)$ and $t>-\frac{1}{4m_2}\alpha_n H$
\[
u(x,t)\le \left(1-\frac{\alpha_n}{4}\right)H-tm_2\le H.
\]
Thus
\[
Q_{\left(1-\frac{\alpha_n}{4}\right)H}(0)\times\left[-\frac{\alpha_nH}{4m_2},0\right]\subset Q_H.
\]
Combined with Lemma \ref{hr}, we obtain \eqref{parabolic case level set} via the scaling $\Gamma_H$.
\end{proof}

Define the rescaled solution
\[
u_H\left(y,s\right):=R^{-2}u\left(\Gamma_{H}^{-1}\left(y,s\right)\right),\quad (y,s)\in Q_H^{*},
\]
which satisfies the normalized problem:
\[
\left\{\begin{array}{ll}
-u_{H,s} \operatorname{det} D_y^{2} u_H=f\circ\Gamma_{H}^{-1} & \text { in } Q_{H}^{*}, \\
u_H=\frac{H}{R^{2}}  & \text { on } \partial_{p} Q_{H}^{*}, \\
m_1 \leq -u_{H,s} \leq m_2 & \text { in } Q_{H}^{*} .
\end{array}\right.
\]
By adapting the techniques in the proof of Theorem 3.2 in \cite{ww92}, there exists uniquely a parabolically convex solution $\overline{\xi}\in C^0(\overline{Q_H^{*}})\cap C^{\infty}(Q_H^{*})$ to the problem
\[
\left\{\begin{array}{ll}
-\overline{\xi}_{s} \operatorname{det} D_y^{2} \overline{\xi}=1 & \text { in } Q_{H}^{*}, \\
\overline{\xi}=\frac{H}{R^{2}} & \text { on } \partial_{p} Q_{H}^{*}, \\
m_1 \leq -\overline{\xi}_{s} \leq m_2 & \text { in } Q_{H}^{*} .
\end{array}\right.
\]

We establish the estimate of $\left\|u_H-\overline{\xi}\right\|_{C^0\left(\overline{Q_H^*}\right)}$.

\begin{lemma}\label{cp'}
There exist constants $\gamma\in(0,1)$ and $C>0$ depending only on $n,m_1,m_2,\lambda,\Lambda$, such that for $R>1$
\[
\left\|u_H-\overline{\xi}\right\|_{C^0\left(\overline{Q_H^*}\right)}\le CR^{-\gamma}.
\]
\end{lemma}

\begin{proof}
For $e\in E$, define the rescaled vector
\begin{equation}\label{tildee}
\tilde{e}:=\frac{a_He}{R}. 
\end{equation}
By 
\[
\frac{\lambda}{m_2}\le\operatorname{det} D_y^{2} u_H\left(y,0\right)=\frac{f\circ\Gamma_{H}^{-1}\left(y,0\right)}{-u_{H,s}\left(y,0\right)}\le\frac{\Lambda}{m_1},\quad y\in Q_{H}^{*}(0)
\]
and Lemma 2.1 in \cite{cl04}, there exist $\alpha, C>0$, such that 
\[
|\tilde{e}|\le CR^{-\alpha}|e|. 
\]
Let $\eta:=\min\left\{2\alpha,2\right\}$. Then
\[
\left(\sum_{i}\left\|\tilde{e_i}\right\|^2+R^{-2}\right)\le CR^{-\eta}\left(\sum_{i}\left\|e_i\right\|^2+1\right)\le CR^{-\eta},\quad R>1.
\]
Observe the normalization
\[
-\mkern-19mu\int_{\frac{a_H\left([-1,1]^n\right)}{R}}f_1\left(Ra_H^{-1}y\right)dy=-\mkern-19mu\int_{[-R^{-2},0]}f_2\left(R^2s\right)ds=1.
\]
Lemma \ref{homogenization} implies that
\[
\left\|u_H-\overline{\xi}\right\|_{C^0\left(\overline{Q_H^*}\right)}\le C\left(\sum_{i}\left\|\tilde{e_i}\right\|^2+R^{-2}\right)^\beta \le CR^{-\beta\eta}, \quad R>1
\]
for $C>0$ independent of $H$. Taking $\gamma:=\beta\eta$ completes the proof. 
\end{proof}
Assume
\[
\overline{\xi}\left(\overline{y},0\right)=\inf\limits_{Q_H^*}\overline{\xi}.
\]
We present two technical lemmas essential for the nonlinear perturbation method. The first step establishes a critical estimate for $\left|\overline{y}\right|$.

\begin{lemma}\label{lem1}
There exist constants $\overline{H},C$ depending only on $n,\lambda,\Lambda,m_1,m_2$, such that for all $H\ge\overline{H}$
\[
|\overline{y}|\le CR^{-\frac{\gamma}{2}}.
\]
\end{lemma}

\begin{proof}
Consider the second-order Taylor expansion of $\overline{\xi}\left(\cdot, 0\right)$ at $\overline{y}$
\begin{equation}\label{eqa}
\overline{\xi}\left(0_n,0\right)=\overline{\xi}\left(\overline{y}, 0\right)+\frac{1}{2}\overline{y}'D^2_y\overline{\xi}\left(y_1, 0\right)\overline{y},
\end{equation}
where $y_1=\theta\overline{y}$ for some $\theta\in(0,1)$. 

In order to apply Lemma \ref{lode}, we establish an $H$-independent lower bound for $\text{dist}_p\left(\left(y_1,0\right),\partial_pD\right)$. By the convexity of $\overline{\xi}$ in $y$
\begin{equation}\label{eqb}
\overline{\xi}\left(y_1, 0\right)\le \theta\overline{\xi}\left(\overline{y}, 0\right)+(1-\theta)\overline{\xi}\left(0_n,0\right).
\end{equation}
It follows from Lemma \ref{cp'} that
\begin{equation}\label{eqd}
\left|\overline{\xi}\left(0_n,0\right)\right|\le CR^{-\gamma}.
\end{equation}
This, combined with Lemma \ref{cp'}, implies that
\[
CR^{-\gamma}=u_H\left(0_n,0\right)-CR^{-\gamma}\le u_H\left(\overline{y}, 0\right)-CR^{-\gamma}\le\overline{\xi}\left(\overline{y}, 0\right)\le\overline{\xi}\left(0_n,0\right)\le CR^{-\gamma},
\]
which means
\begin{equation}\label{eqc}
\left|\overline{\xi}\left(\overline{y}, 0\right)\right|\le CR^{-\gamma}.
\end{equation}
For sufficiently large $H$, combining \eqref{eqb},\eqref{eqd} and \eqref{eqc} yields that
\[
\overline{\xi}\left(y_1, 0\right)\le CR^{-\gamma}<\frac{H}{2R^2}.
\]
By \eqref{lower}, this guarantees that
\[
\text{dist}_p\left(\left(y_1,0\right),\partial_pD\right)\ge C>0
\] 
for a constant $C$ independent of $H$. 

Applying Lemma \ref{lode} and \eqref{u_t}, we obtain $H$-independent lower bound 
\[
D_y^2\overline{\xi}\left(y_1, 0\right)\ge C^{-1}I.
\]
Substituting this into \eqref{eqa} and employing \eqref{eqd},\eqref{eqc} gives the estimate
\[
\left|\overline{y}\right|\le CR^{-\frac{\gamma}{2}}.
\]
\end{proof}

The second lemma gives the estimate for $\left\|a_H\right\|$ for a fixed $H$.

\begin{lemma}\label{lem2}
There exists a constant $C$ depending only on $n,\lambda,\Lambda,m_1,m_2,H$, $\sup\limits_{B_1}u(\cdot,0)$, such that $\left\|a_H\right\|\le C$.
\end{lemma}

\begin{proof}
From Step 1 in the proof of Lemma \ref{normalization} we have the inclusion 
\begin{equation}\label{eq0a}
a_{H}\left(Q_{H}(0)\right) \subset B_{2nR}.
\end{equation}
It follows from Proposition 2.8 in \cite{cl03} that for every $H\ge \sup\limits_{B_1}u(\cdot,0)$
\begin{equation}\label{eq0b}
Q_H(0)\subset B_{CH^{\frac{n}{2}}},
\end{equation}
where $C>0$ depending only on $n,\lambda,\Lambda,m_1,m_2$. Combining \eqref{eq0b} and Corollary 4.7 in \cite{f17}, gives a $\rho>0$, depending only on $n,\lambda,\Lambda,m_1,m_2,H$, $\sup\limits_{B_1}u(\cdot,0)$, such that 
\begin{equation}\label{eq0c}
B_\rho\subset Q_{H}(0).
\end{equation}
\eqref{eq0a} and \eqref{eq0b} implies 
\[
a_{H}\left(B_\rho\right)\subset B_{2nR}.
\]
Therefore $\left\|a_H\right\|\le C$.
\end{proof}

Now we give a preliminary perturbation estimate. First, we define for $\tilde{H}\in(0,H]$
\[
\begin{array}{l}
\mathcal{E}_{\tilde{H}}:=\left\{(y, s) \in \mathbb{R}_{-}^{n+1}: \frac{1}{2} y' D_y^2 \overline{\xi}\left(\overline{y}, 0\right) y+\overline{\xi}_s\left(\overline{y}, 0\right) s<\tilde{H}\right\}, \\
\mathcal{E}_{\tilde{H}}\left(\overline{y}, 0\right):=\left\{(y, s) \in \mathbb{R}_{-}^{n+1}: \frac{1}{2}\left(y-\overline{y}\right)' D_y^2 \overline{\xi}\left(\overline{y}, 0\right)\left(y-\overline{y}\right)+\overline{\xi}_s\left(\overline{y}, 0\right) s<\tilde{H}\right\}, \\
\mathcal{S}_{\tilde{H}}\left(\overline{y}, 0\right):=\left\{(y, s) \in \mathbb{R}_{-}^{n+1}: \frac{1}{2}\left(y-\overline{y}\right)' D_y^2 \overline{\xi}\left(\overline{y}, 0\right)\left(y-\overline{y}\right)+\overline{\xi}_s\left(\overline{y}, 0\right) s=\tilde{H}\right\}.
\end{array}
\]

\begin{lemma}\label{ppe}
Assume $\epsilon\in\left(0,\frac{\gamma}{4-\gamma}\right)$. Then there exist $\overline{k}$ and $\overline{C}$ depending only on $n,\lambda,\Lambda,m_1,m_2,\epsilon$, such that for $k\in\mathbb{N}$ with $k>\overline{k}$, $H=2^{\left(1+\epsilon\right)k}$ and $2^{k-1}\le H'\le 2^k$
\[
\mathcal{E}_{\frac{H^{\prime}}{R^2}-\overline{C} 2^{-\frac{3 \epsilon k}{2}}} \subset \Gamma_H\left(Q_{H^{\prime}}\right) \subset
\mathcal{E}_{\frac{H^{\prime}}{R^2}+\overline{C} 2^{-\frac{3 \epsilon k}{2}}}.
\]
\end{lemma}

\begin{proof}
\textbf{Step 1: Obtain estimates for higher order derivatives.} Combining Lemma \ref{normalization} and \ref{lem1} implies that for every $H\ge\overline{H}$
\begin{equation}\label{overliney}
\text{dist}_p\left(\left(\overline{y},0\right),\partial_pD\right)\ge C>0
\end{equation}
for an $H$-independent constant $C$. This, together with \eqref{u_t} and Lemma \ref{lode} implies the lower bound
\begin{equation}\label{hessianlowerbound}
D_y^2\overline{\xi}\left(\overline{y},0\right)\ge C^{-1}I.
\end{equation}
For $\overline{C}$ to be determined and $(y,s)\in \mathcal{E}_{\frac{H^{\prime}}{R^2}-\frac{\overline{C}}{2} 2^{-\frac{3 \epsilon k}{2}}}\left(\overline{y}, 0\right)$, \eqref{u_t} and \eqref{hessianlowerbound} yield that
\begin{equation}\label{eq001}
\left|(y,s)-\left(\overline{y}, 0\right)\right|_p^2\le C\left(\frac{H'}{R^2}-\frac{\overline{C}}{2}2^{-\frac{3\epsilon k}{2}}\right).
\end{equation} 
By Lemma \ref{hr}, we observe that 
\begin{equation}\label{h'r}
C^{-1}2^{-\epsilon k}\le\frac{H^{\prime}}{R^2}\le C2^{-\epsilon k}.
\end{equation}
Thus for $k$ large, combining \eqref{overliney}, \eqref{eq001} and \eqref{h'r}, there exists an $H$-independent constant $C>0$, such that
\[
\text{dist}_p\left(\left(y,s\right),\partial_pD\right)\ge C>0
\]
By Lemma \ref{lode}, we deduce that for large $\overline{k}$
\[
\left|D^2_y\overline{\xi}(y,s)\right|\le C.
\] 
Applying Lemma \ref{uhode}, we obtain the uniform bounds
\begin{equation}\label{hode'}
\left|D_y^3\overline{\xi}(y, s)\right|+\left|D_y\overline{\xi}_s(y,s)\right|+\left|\overline{\xi}_{ss}(y,s)\right|\le C.
\end{equation}

\textbf{Step 2: Derive the left inclusion by Taylor expansions.} We apply \eqref{hode} to obtain the Taylor expansion
\begin{equation}\label{eq002}
\begin{aligned}
\overline{\xi}(y,s)&=\overline{\xi}\left(\overline{y}, 0\right)+\overline{\xi}_s\left(\overline{y}, 0\right)s+\frac{1}{2}\left(y-\overline{y}\right)' D_y^2 \overline{\xi}\left(\overline{y}, 0\right)\left(y-\overline{y}\right)\\
&\quad+O\left(\left|(y,s)-\left(\overline{y}, 0\right)\right|_p^3\right).
\end{aligned}
\end{equation}
We observe that for $\epsilon\in\left(0,\frac{\gamma}{4-\gamma}\right)$, Lemma \ref{hr} implies that for large $k$
\begin{equation}\label{epsilon}
R^{-\gamma}\le C2^{-\frac{1+\epsilon}{2}k\gamma}<<2^{-\frac{3\epsilon k}{2}}.
\end{equation}
For sufficiently large $k,\overline{C}$, combining \eqref{eqc}, \eqref{eq001}, \eqref{hode'} and \eqref{epsilon} yields 
\[
\overline{\xi}(y,s)\le \overline{\xi}\left(\overline{y}, 0\right)+\frac{H'}{R^2}-\frac{\overline{C}}{2}2^{-\frac{3\epsilon k}{2}}+C2^{-\frac{3\epsilon k}{2}}\le \frac{H'}{R^2}- CR^{-\gamma}.
\]
Employing Lemma \ref{cp'}, we obtain that
\begin{equation}\label{eq004}
\mathcal{E}_{\frac{H^{\prime}}{R^2}-\frac{\overline{C}}{2} 2^{-\frac{3 \epsilon k}{2}}}\left(\overline{y}, 0\right)\subset \left\{u_H<\frac{H'}{R^2}\right\}.
\end{equation}

\textbf{Step 3: Derive the right inclusion.} To establish the other inclusion
\begin{equation}\label{eq005}
\left\{u_H<\frac{H'}{R^2}\right\}\subset \mathcal{E}_{\frac{H^{\prime}}{R^2}+\frac{\overline{C}}{2} 2^{-\frac{3 \epsilon k}{2}}}\left(\overline{y}, 0\right),
\end{equation}
it suffices to show
\begin{equation}\label{eq006}
\mathcal{S}_{\frac{H^{\prime}}{R^2}+\frac{\overline{C}}{2} 2^{-\frac{3 \epsilon k}{2}}}\left(\overline{y}, 0\right)\subset \left\{u_H\ge \frac{H'}{R^2}\right\}.
\end{equation}
For large $k$ and every $(y,s)\in\mathcal{S}_{\frac{H^{\prime}}{R^2}+\frac{\overline{C}}{2} 2^{-\frac{3 \epsilon k}{2}}}\left(\overline{y}, 0\right)$, similar to \eqref{eq001}, we have the estimate for the parabolic distance from $(y,s)$ to $\left(\overline{y},0\right)$:
\begin{equation}\label{eq001r}
\left|(y,s)-\left(\overline{y}, 0\right)\right|_p^2\le C\left(\frac{H'}{R^2}+\frac{\overline{C}}{2}2^{-\frac{3\epsilon k}{2}}\right).
\end{equation} 
Estimates for derivatives of $\overline{\xi}$ then follow Lemma \ref{lode} and \ref{uhode}
\begin{equation}\label{hode''}
\left|D_y^2\overline{\xi}(y, s)\right|+\left|D_y^3\overline{\xi}(y, s)\right|+\left|D_y\overline{\xi}_s(y,s)\right|+\left|\overline{\xi}_{ss}(y,s)\right|\le C.
\end{equation}
It follows from \eqref{eqc}, \eqref{epsilon}, \eqref{eq001r} and \eqref{hode''} that for $k,\overline{C}$ large enough
\[
\overline{\xi}(y,s)\ge \overline{\xi}\left(\overline{y}, 0\right)+\frac{H'}{R^2}+\frac{\overline{C}}{2}2^{-\frac{3\epsilon k}{2}}-C2^{-\frac{3\epsilon k}{2}}\ge \frac{H'}{R^2}+ CR^{-\gamma}.
\]
Applying Lemma \ref{cp'}, we obtain that
\[
u_H(y,s)\ge \overline{\xi}(y,s)-CR^{-\gamma}\ge \frac{H'}{R^2},
\]
which gives \eqref{eq006} and then \eqref{eq005}. 

\textbf{Step 4: Apply the estimate for $\left|\overline{y}\right|$.} Combining \eqref{eq004} and \eqref{eq005}, we attain that
\begin{equation}\label{eq007}
\mathcal{E}_{\frac{H^{\prime}}{R^2}-\frac{\overline{C}}{2} 2^{-\frac{3 \epsilon k}{2}}}\left(\overline{y}, 0\right) \subset \Gamma_H\left(Q_{H^{\prime}}\right) \subset
\mathcal{E}_{\frac{H^{\prime}}{R^2}+\frac{\overline{C}}{2} 2^{-\frac{3 \epsilon k}{2}}}\left(\overline{y}, 0\right).
\end{equation}
Using Lemma \ref{lem1}, $\epsilon\in\left(0,\frac{\gamma}{4-\gamma}\right)$ and \eqref{eq007}, we refine the inclusions
\[
\mathcal{E}_{\frac{H^{\prime}}{R^2}-\overline{C} 2^{-\frac{3 \epsilon k}{2}}}\subset \mathcal{E}_{\frac{H^{\prime}}{R^2}-\frac{\overline{C}}{2} 2^{-\frac{3 \epsilon k}{2}}}\left(\overline{y}, 0\right),\quad \mathcal{E}_{\frac{H^{\prime}}{R^2}+\frac{\overline{C}}{2} 2^{-\frac{3 \epsilon k}{2}}}\left(\overline{y}, 0\right)\subset \mathcal{E}_{\frac{H^{\prime}}{R^2}+\overline{C} 2^{-\frac{3 \epsilon k}{2}}}.
\]
This completes the proof of the lemma.
\end{proof}

Now we iteratively apply the pertubation estimate in Proposition \ref{ppe}  to capture the behavior of $u$ at infinity.

\begin{prop}\label{prop0}
There exist an upper-triangular matrix $T$, $\tau>0$ with $\tau\operatorname{det}T^2=1$, and a constant $\overline{C}>1$ depending only on $n,m_1,m_2,\lambda,\Lambda,\epsilon$, $\sup\limits_{B_1}u(\cdot,0)$, such that 
\[
\overline{C}^{-1}I\le T\le \overline{C}I, \quad m_1\le \tau\le m_2,
\]
\[
\left| w(y,s)-\frac{1}{2}|y|^{2}+s \right|\le \overline{C}\left( |y|^{2}-s \right)^{\frac{4-\epsilon}{4}} \quad\text{for}\quad |y|^2-s\ge \overline{C},
\]
where $w(y,s):=u\left(T^{-1}y,\tau^{-1}s\right)$.
\end{prop}

\begin{proof}
\textbf{Step 1: A primary level set version. }From Lemma \ref{ppe}, for sufficiently large $k$, we have the following containment relation
\[
\mathcal{E}_{\big(1-C 2^{-\frac{\epsilon k}{2}}\big)H'} \subset \begin{pmatrix}a_H &0_n\\0_n'&1\end{pmatrix}Q_{H'} \subset
\mathcal{E}_{\big(1+C 2^{-\frac{\epsilon k}{2}}\big)H'}.
\]
Denote $Q_k$ as the $n\times n$ symmetric positive definite matrix satisfying $Q_k^2=D_y^2\overline{\xi}\left(\overline{y}, 0\right)$. Since $Q_ka_H$ is invertible, we apply the Gram-Schmit process to the column vectors of $Q_ka_H$. Then there exists uniquely an orthogonal matrix $O_k$, such that $T_k:=O_kQ_ka_H$ is upper-triangular. Define 
\[
\Sigma_k:=\begin{pmatrix}T_k &0_n\\0_n'&\tau_k\end{pmatrix},
\] 
where $\tau_k:=-\overline{\xi}_s\left(\overline{y}, 0\right)$ for $H=2^{\left(1+\epsilon\right)k}$. 
This yields the containment
\begin{equation}\label{eq008}
E_{\big(1-C 2^{-\frac{\epsilon k}{2}}\big)H'}\subset\Sigma_k Q_{H'}\subset E_{\big(1+C 2^{-\frac{\epsilon k}{2}}\big)H'}
\end{equation}
for $2^{k-1}\le H'\le 2^k$. Take $H=2^{\left(1+\epsilon\right)k}$ and $H'=2^{k-1}$, and then
\begin{equation}\label{eq-k}
E_{\big(1-C 2^{-\frac{\epsilon k}{2}}\big)2^{k-1}}\subset\Sigma_k Q_{2^{k-1}}\subset E_{\big(1+C 2^{-\frac{\epsilon k}{2}}\big)2^{k-1}}.
\end{equation}
While for $H=2^{\left(1+\epsilon\right)(k-1)}$ and $H'=2^{k-1}$
\begin{equation}\label{eq-k-1}
E_{\big(1-C 2^{-\frac{\epsilon (k-1)}{2}}\big)2^{k-1}}\subset\Sigma_{k-1} Q_{2^{k-1}}\subset E_{\big(1+C 2^{-\frac{\epsilon (k-1)}{2}}\big)2^{k-1}},
\end{equation}
which immediately implies that
\begin{equation}\label{eq-k-1'}
\Sigma_{k-1}^{-1}E_{\big(1-C 2^{-\frac{\epsilon (k-1)}{2}}\big)2^{k-1}}\subset Q_{2^{k-1}}\subset \Sigma_{k-1}^{-1}E_{\big(1+C 2^{-\frac{\epsilon (k-1)}{2}}\big)2^{k-1}}.
\end{equation}
It follows from \eqref{eq-k} and \eqref{eq-k-1'} that for a larger constant $C$
\begin{equation}\label{primaryversion}
E_{1-C 2^{-\frac{\epsilon k}{2}}}\subset\Sigma_k\Sigma_{k-1}^{-1} E_1\subset E_{1+C 2^{-\frac{\epsilon k}{2}}}.
\end{equation}

\textbf{Step 2: A uniform estiamte for $\left\|\Sigma_k\right\|$. }Applying Lemma 2.1 in \cite{zbw18} to \eqref{primaryversion}, we obtain the estimate
\[
\left\|\Sigma_k\Sigma_{k-1}^{-1}-I_{n+1}\right\|\le C2^{-\frac{\epsilon k}{4}}.
\]
This immediately gives 
\[
\left\|\Sigma_k\Sigma_{k-1}^{-1}\right\|\le 1+C2^{-\frac{\epsilon k}{4}}. 
\]
For $k>\overline{k}+1$, we derive the following bound 
\begin{equation}\label{preestimate}
\left\|\Sigma_k\Sigma_{\overline{k}}^{-1}-\Sigma_{k-1}\Sigma_{\overline{k}}^{-1}\right\|\le C2^{-\frac{\epsilon k}{4}}\prod\limits_{i=\overline{k}+1}^{k-1}\left(1+C2^{-\frac{\epsilon i}{4}}\right)\le C2^{-\frac{\epsilon k}{4}},
\end{equation}
the last inequality of which follows from
\[
\prod\limits_{i=\overline{k}+1}^{k-1}\left(1+C2^{-\frac{\epsilon i}{4}}\right)=e^{\sum\limits_{i=\overline{k}+1}^{k-1}\ln\left(1+C2^{-\frac{\epsilon i}{4}}\right)}\le e^{C\sum\limits_{i=\overline{k}+1}^{k-1}2^{-\frac{\epsilon i}{4}}}\le C.
\]
From Lemma \ref{lem2}, we have the bound for $k=\overline{k}$
\[
\left\|a_{2^{\overline{k}}}\right\|\le C\left(\overline{k}\right).
\]
It follows that 
\[
\sum\limits_{i=0}^{\overline{k}}\left\|\Sigma_{i}\right\|\le C\left(\overline{k}\right)
. 
\]
Combining this and \eqref{preestimate} yields the estimate
\begin{equation}\label{interval}
\left\|\Sigma_k-\Sigma_{k-1}\right\|\le C2^{-\frac{\epsilon k}{4}}.
\end{equation}
Hence there exist an upper-triangular matrix $T$ and $m>0$ such that as $k\rightarrow\infty$
\[
T_k\rightarrow T\quad\text{in}\quad\mathbb{R}^{n\times n},\quad \tau_k\rightarrow \tau \quad\text{in}\quad\left(-\infty,0\right].
\]
This implies that 
\[
\tau\det T^2=1, \quad m_1\le \tau\le m_2
\] 
because 
\[
\tau_k\det T_k^2=1,\quad m_1\le \tau_k\le m_2. 
\]
Define 
\[
\Sigma:=\begin{pmatrix}T &0\\0'&m\end{pmatrix}.
\]
It follows from \eqref{interval} that
\[
\left\|\Sigma_k-\Sigma\right\|\le C2^{-\frac{\epsilon k}{4}}. 
\]
This implies the uniform bound
\begin{equation}\label{T_k}
\left\|\Sigma_k\right\|+\left\|\Sigma_k^{-1}\right\|\le C 
\end{equation}
for a constant $C$ independent of $k$. 

\textbf{Step 3: A refined level set estimate.} It follows from \eqref{eq008} that 
\begin{equation}\label{roughinclusion}
\Sigma_k Q_{H'}\subset E_{CH'}.
\end{equation} 
For every $(x,t)\in Q_{H'}$ with $H'\in\left[2^{k-1},2^k\right]$, the following estimate holds
\[
\left|\Sigma(x,t)-\Sigma_k(x,t)\right|_p\le \left\| \Sigma- \Sigma_k \right\| \left\|\Sigma_k^{-1}\right\| \left|\Sigma_k(x,t)\right|_p\le CH'^{\frac{1}{2}}2^{-\frac{\epsilon k}{4}},
\]
where \eqref{roughinclusion} is employed in the last inequality. By combining this with \eqref{eq008} and possibly enlarging $C$, we establish the refined containment
\begin{equation}\label{eq009}
E_{\big(1-CH'^{-\frac{\epsilon}{4}}\big)H'}\subset Q_{H'}(w)=\Sigma Q_{H'}\subset E_{\big(1+CH'^{-\frac{\epsilon}{4}}\big)H'},
\end{equation}
where $w:=u\circ \Sigma^{-1}$ and $Q_{H'}(w):=\left\{(y,s)\in\mathbb{R}^{n+1}_-:w(y,s)<H'\right\}$. 

\textbf{Step 4: Conversion to the solution-estimate version. }By \eqref{eq009}, there exists $\overline{H}$, such that for $H'\ge\overline{H}$ and $(y,s)\in \partial_p\left(Q_{H'}(w)\right)$
\begin{equation}\label{onehand}
\frac{1}{2}|y|^2-s\ge\left(1-CH'^{-\frac{\epsilon}{4}}\right)H'\ge\frac{H'}{2},
\end{equation}
which implies that
\begin{equation}\label{Hover2}
|y|^2-s\ge\frac{H'}{2}. 
\end{equation}
Consequently, we derive from \eqref{onehand} the upper bound 
\[
\begin{aligned}
w(y,s)=H'&\le \frac{1}{2}|y|^2-s + CH'^{1-\frac{\epsilon}{4}}\\
&\le \frac{1}{2}|y|^2-s + C\left(|y|^2-s\right)^{1-\frac{\epsilon}{4}}.
\end{aligned}
\]
On the other hand for $H'\ge\overline{H}$ and $(y,s)\in \partial_p\left(Q_{H'}(w)\right)$
\begin{equation}\label{otherhand}
\frac{1}{2}|y|^2-s\le \left(1+CH'^{-\frac{\epsilon}{4}}\right)H'\le \frac{3H'}{2}.
\end{equation}
Using \eqref{Hover2} and \eqref{otherhand}, we obtain the complementary lower bound
\[
\begin{aligned}
w(y,s)=H'&\ge\frac{1}{2}|y|^2-s-CH'^{1-\frac{\epsilon}{4}}\\ &\ge\frac{1}{2}|y|^2-s - C\left(|y|^2-s\right)^{1-\frac{\epsilon}{4}}.
\end{aligned}
\]
\eqref{otherhand} also yields that 
\[
|y|^2-s\le 3H'
\] 
and furthermore
\[
Q_{\overline{H}}(w)\subset\left\{(y,s)\in\mathbb{R}^{n+1}_-:|y|^2-s\le 3\overline{H}\right\}. 
\]
This implies that
\[
\left\{(y,s)\in\mathbb{R}^{n+1}_-:|y|^2-s> 3\overline{H}\right\}\subset\bigcup\limits_{H\ge\overline{H}}\partial_p\left(Q_{H'}(w)\right).
\]
As a result, we obtain that for $|y|^2-s> 3\overline{H}$
\[
\left|w(y,s)-\frac{1}{2}|y|^2+s\right|\le C\left(|y|^2-s\right)^{1-\frac{\epsilon}{4}}.
\]
\end{proof}

According to \eqref{T_k}, we establish the following uniform bound for every $i\in\mathbb{N}$
\begin{equation}\label{a_H}
\left\|a_{2^{\left(1+\epsilon\right)k}}\right\|+\left\|a_{2^{\left(1+\epsilon\right)k}}^{-1}\right\|\le C,
\end{equation}
where $C$ is a positive constant depending only on $n,m_1,m_2,\lambda,\Lambda,\epsilon$, $\sup\limits_{B_1}u(\cdot,0)$.

From Proposition \ref{prop0}, we derive the asymptotic behavior of $u$. There exists a constant $C>0$ depending only on $n,m_1,m_2,\lambda,\Lambda,\epsilon$, $\sup\limits_{B_1}u(\cdot,0)$, such that
\begin{equation}\label{asymptotic}
\left| u(x,t)-\frac{1}{2}x'Ax+\tau t \right|\le C\left( |x|^{2}-t \right)^{\frac{4-\epsilon}{4}}\quad \text{for}\quad |x|^2-t\ge C,
\end{equation}
where $A:=T'T$.

\quad

\section{Positive upper and lower bounds of $D_x^2u$}\label{3'}
We introduce the set
\[
E:=\left\{k_1e_1+\cdots+k_ne_n:k_1,\cdots,k_n\in\mathbb{Z},k_1^2+\cdots+k_n^2>0\right\}.
\]
For any $(x,t)\in\mathbb{R}^{n+1}_-$ and $e\in E$, we define the second order diference quotient
\[
\Delta_e^2u(x,t):=\frac{u(x+e,t)+u(x-e,t)-2u(x,t)}{|e|^2}.
\]
Define the function $F\left(a,M\right):=\log a+\log\det M$ for $a>0$ and $n\times n$ symmetric positive definite matrix $M$. The concavity of $F$ yields that $\Delta_e^2u$ is a subsolution.

\begin{lemma}\label{concave}
For any $e\in E$, $\Delta_e^2u$ satisfies
\[
\frac{1}{u_t}D_t\Delta_e^2u+u^{ij}D_{x_ix_j}\Delta_e^2u\ge0.
\]
\end{lemma}

\begin{proof}
Let $w(x,t):=\frac{1}{2}\left[u(x+e,t)+u(x-e,t)\right]$. By the concavity of $F$
\[
\begin{aligned}
F\left( -w_t,D_x^2w \right)
&\ge\frac{1}{2} \left[ F\left(-u_t(x+e,t),D_x^2u(x+e,t)\right) + F\left(-u_t(x-e,t),D_x^2u(x-e,t)\right) \right]\\
&=\frac{1}{2}\left[\log f(x+e,t)+\log f(x-e,t)\right]=\log f(x,t).
\end{aligned}
\]
On the other hand, we have from the concavity of $F$ that
\[
\begin{aligned}
F\left(-w_t,D_x^2w\right)&\le F\left(-u_t,D_x^2u\right)-\frac{\partial F}{\partial a}\left(-u_t,D_x^2u\right)(w-u)_t\\
&\quad+\frac{\partial F}{\partial M_{ij}}\left(-u_t,D_x^2u\right)D_{ij}(w-u)\\
&=\log f(x,t)+\left(\frac{1}{u_t}D_t\Delta_e^2u+u^{ij}D_{x_ix_j}\Delta_e^2u\right)\frac{|e|^2}{2}.
\end{aligned}
\]
Combining the two inequalities above, we get the result.
\end{proof}

We recall a calculas lemma.

\begin{lemma}\label{a.1}[Lemma A.1 in \cite{cl04}]
Let $g\in C^2(-1,1)$ be a strictly convex function, and let $0<|h|\le\epsilon$. Then
\[
\Delta_h^2g(x)>0,\quad\text{for}\quad |x|\le 1-2\epsilon
\]
and
\[
\int_{-1+2\epsilon}^{1-2\epsilon}\Delta_h^2g\le \frac{C}{\epsilon}\mathop{\text{osc}}\limits_{(-1.1)}g,
\]
where $C$ is some universal constant and $\mathop{\text{osc}}\limits_{(-1.1)}g:=\sup\limits_{-1<s<t<1}|g(s)-g(t)|$. 
\end{lemma}

Building upon Lemma \ref{concave}, the local maximum principle and the preceding lemma, we establish the following local boundedness of $\Delta_{\tilde{e}}^2u_H$.

\begin{lemma}\label{bdd1}
For $r>0$ and $e\in E$, there exist constants $C,H_0$ depending only on $n,\lambda,\Lambda,m_1,m_2$, $r,|e|$, such that for all $H\ge H_0$, the following estimates hold
\[
\int_{\left\{Y\in Q_H^*:\text{dist}_p\left(Y,\partial_pQ_H^*\right)>\frac{r}{2}\right\}}\Delta_{\tilde{e}}^2u_H\le C,
\]
and 
\[
0<\Delta_{\tilde{e}}^2u_H\le C
\]
for every $Y\in Q_H^*$ satisfying $\text{dist}_p\left(Y,\partial_pQ_H^*\right)>r$, where $\tilde{e}$ is defined by \eqref{tildee}.
\end{lemma}

\begin{proof}
It follows from the strict convexity of $u$ that $\Delta_{\tilde{e}}^2u_H>0$. We have that $|\tilde{e}|\le CR^{-\alpha}|e|\rightarrow0$ as $H\rightarrow\infty$. Thus there exists $H_0$ such that for $H>H_0$, $|\tilde{e}|\le\frac{r}{4}$. Let $L$ be a line parallel to $\tilde{e}$, then we have by Lemma \ref{a.1} that
\[
\int_{L\cap \left\{ Y\in Q_H^*:\text{dist}_p\left(Y,\partial_pQ_H^*\right)>\frac{r}{2} \right\} }\Delta_{\tilde{e}}^2u_H dL\le C,
\]
where the constant $C$ is independent of $H$. Integrating the above over all such lines and $t$ variable, we can obtain that
\[
\int_{\left\{Y\in Q_H^*:\text{dist}_p\left(Y,\partial_pQ_H^*\right)>\frac{r}{2}\right\}}\Delta_{\tilde{e}}^2u_H dY\le C.
\]

By Lemma \ref{concave}, we know that
\[
\left(D_s u_H\right)^{-1}D_s\Delta_{\tilde{e}}^2u_H+u_H^{ij}D_{y_iy_j}\Delta_{\tilde{e}}^2u_H\ge0.
\]
Theorem 1.11 in \cite{zb18} gives that $\Delta_{\tilde{e}}^2u_H\left(Y\right)\le C$ for $\text{dist}_p\left(Y,\partial_pQ_H^*\right)>r$.
\end{proof}

The boundedness of $\Delta_e^2u$ in $\mathbb{R}^{n+1}_-$ then follows.

\begin{lemma}\label{bdd2}
$\gamma:=\sup\limits_{e\in E,X\in\mathbb{R}^{n+1}_-}\Delta_e^2u\left(X\right)<\infty$.
\end{lemma}

\begin{proof}
For $e\in E$ and $X\in\mathbb{R}^{n+1}_-$, let $Y:=\Gamma_HX$. We take $H=2^{\left(1+\epsilon\right)k}$ large so that $X\in Q_{\frac{H}{2}}$. Applying Lemma \ref{lemma} and \ref{hr} to equation 
\[
-u_{H,s} \operatorname{det} D_y^{2} u_H=f\circ\Gamma_{H}^{-1},
\]
we have that 
\[
\text{dist}_p\left(Y,\partial_pQ_H^*\right)\ge C>0,
\]
where $C$ is independent of $H$. Therefore we attain that
\[
\begin{aligned}
\Delta_e^2u(x,t)&=\frac{u(x+e,t)+u(x-e,t)-2u(x,t)}{|e|^2}\\
&=\frac{|a_He|^2}{|e|^2}\Delta_{\tilde{e}}^2u_H\le C_1\|a_H\|^2\le C_2,
\end{aligned}
\]
where we have employed Lemma \ref{bdd1} in the penultimate inequality and \eqref{a_H} in the last one.
\end{proof}

To establish positive upper and lower bounds of $D_x^2u$, we need the following estimate.

\begin{lemma}\label{cl03}[Lemma 2.9 in \cite{cl03}]
Let $n\ge2$, $\lambda>0$, $r\ge2$, and $u$ be a $C^2$ function in $(-3,3)^{n-1}\times(-r,r)$ satisfying
\[
D^2_xu>0,\quad \det D^2_xu\ge\lambda,\quad\text{in}\quad (-3,3)^{n-1}\times(-r,r)
\]
and
\[
0\le u\le1 \quad\text{in}\quad(-2,2)^{n-1}.
\]
Then for some positive constant $C=C(n)>0$,
\[
\max\limits_{|s|\le r}u\left(0_{n-1},s\right)^n\ge \left( \frac{r\lambda}{C}-1 \right).
\]
\end{lemma}

\begin{prop}\label{bdd}
There exists $C>0$, such that $C^{-1}I\le D_x^2u\le CI$ in $\mathbb{R}^{n+1}_-$.
\end{prop}

\begin{proof}
For fixed $(x,t)\in \mathbb{R}^{n+1}_-$, let
\[
\tilde{u}(y,s):=u(x+y,t+s)-u(x,t)-D_xu(x,t)y.
\]
Then 
\[
\tilde{u}(0,0)=0,\quad \tilde{u}\ge0\quad\text{in}\quad\mathbb{R}^{n+1}_-.
\]

\textbf{Step 1: Estimate $\tilde{u}(\cdot,0)$ on $\partial B_r$ for certain $r>0$. }We will prove that there exist constants $\nu,r>0$, such that 
\[
\nu\le\tilde{u}(\cdot,0)\le \nu^{-1}\quad \text{on}\quad \partial B_r.
\]
From Lemma \ref{bdd1}, there is
\[
\sup_{e\in E}\Delta_e^2\tilde{u}(0,0)=\sup_{e\in E}\Delta_e^2u(x,t)\le\sup_{(\textbf{x},\textbf{t})\in\mathbb{R}^{n+1}_-,e\in E}\Delta_e^2u(\textbf{x},\textbf{t})\le\gamma.
\]
It follows that
\begin{equation}\label{upperbound}
\sup_{y\in \partial B_r}\tilde{u}(y,0)\le C\gamma r^2 \quad\text{for}\quad 1\le r<\infty.
\end{equation}
On the other hand for $\overline{y}\in\partial B_r$, we have from $\sup\limits_{e\in E}\Delta_e^2\tilde{u}\left(\frac{\overline{y}}{2},0\right)\le \gamma$ that
\begin{equation}\label{aninequality}
\tilde{u}\left(\frac{\overline{y}}{2}+e,0\right)+\tilde{u}\left(\frac{\overline{y}}{2}-e,0\right)\le 2\tilde{u}\left(\frac{\overline{y}}{2},0\right)+\gamma|e|^2.
\end{equation}
The convexity of $\tilde{u}$ gives that 
\[
2\tilde{u}\left(\frac{\overline{y}}{2},0\right)\le \tilde{u}\left(\overline{y},0\right)+\tilde{u}\left(0,0\right)=\tilde{u}\left(\overline{y},0\right).
\]
Combined with \eqref{aninequality}, we arrive at
\[
\tilde{u} \left( \frac{\overline{y}}{2}+y,0\right)\le\tilde{u}\left(\overline{y},0\right)+C_0\gamma\quad\text{for}\quad y\in (-2,2)^n.
\]
Taking $\frac{\overline{y}}{\left|\overline{y}\right|}$ as $e_n$ and applying Lemma \ref{cl03} to 
\[
\frac{\tilde{u} \left( \frac{\overline{y}}{2}+\cdot,0 \right)}{\tilde{u}\left(\overline{y},0\right)+C_0\gamma},
\] 
we have that
\[
\tilde{u}\left(\overline{y},0\right)^n
=
\sup_{|s|\le\frac{\left|\overline{y}\right|}{2}}\tilde{u}\left(\frac{\overline{y}}{2}+s\frac{\overline{y}}{\left|\overline{y}\right|},0 \right)^n
\ge
\frac{r\lambda}{C_nm_2}
-
\left(\tilde{u}\left(\overline{y},0\right)+C_0\gamma\right)^n,
\]
where the first equality comes from the convexity of 
Notice that the above inequality is obvious for $n=1$. If $\tilde{u}\left(\overline{y},0\right)\le\gamma$, then 
\[
\tilde{u}\left(\overline{y},0\right)^n\ge\frac{r\lambda}{C_nm_2}
-
\left(1+C_0\right)^n\gamma^n.
\]
We can choose $r>0$ large enough such that $\tilde{u}\left(\overline{y},0\right)\ge1$. Now fix $r$. Thus 
\begin{equation}\label{lowerbound}
\inf_{\overline{y}\in\partial B_r}\tilde{u}\left(\overline{y},0\right)\ge\min\{\gamma,1\}=:\gamma_1.
\end{equation}

\textbf{Step 2: Estimate $\tilde{u}$ on $\partial_pE_{2r}$. }For $(y,s)\in\partial_pE_{2r}$ with $|y|\ge r$, we have that $s\ge-3r^2$. By \eqref{u_t}, \eqref{upperbound} and\eqref{lowerbound}, there is
\[
\tilde{u}\left(y,s\right)\ge\tilde{u}\left(y,0\right)\ge\gamma_1
\]
and
\[
\tilde{u}\left(y,s\right)\le\tilde{u}\left(y,0\right)+3m_2r^2\le(C\gamma+3m_2)r^2.
\]
On the other hand for $(y,s)\in\partial_pE_{2r}$ with $|y|\le r$, we derive that
\[
\tilde{u}\left(y,s\right)\ge\tilde{u}\left(y,0\right)+3m_1r^2\ge3m_1r^2
\]
and
\[
\tilde{u}\left(y,s\right)\le\tilde{u}\left(y,0\right)+4m_2r^2\le(C\gamma+4m_2)r^2.
\]
Thus there exists $\mu>0$, independent of choice of $(x,t)\in\mathbb{R}^{n+1}_-$, such that 
\[
\mu\le \tilde{u}\le\mu^{-1},\quad\text{on}\quad\partial_pE_{2r},
\] 
where the constant $\mu$ is independent of the choice of $(x,t)\in\mathbb{R}_-^{n+1}$.

Applying the $C^{2+\alpha,\frac{2+\alpha}{2}}$ estimate of the parabolic Monge-Ampère equation, which can be derived from the proof of Theorem 4.1 in \cite{ww92}, we have that 
\[
\left|D^2_xu(x,t)\right|=\left|D^2_y\tilde{u}\left(0_n,0\right)\right|\le C,
\] 
where $C$ is independent of the choice of $(x,t)\in\mathbb{R}_-^{n+1}$. Combined with \eqref{pma1} and \eqref{u_t}, there is
\[
C^{-1}I\le D^2_xu(x,t)\le CI.
\]
From the arbitrariness of $(x,t)\in\mathbb{R}^{n+1}_-$, we have proved the proposition.
\end{proof}

\quad

\section{Capture of $\Delta_e^2u$ and $\Delta_k^tu$}\label{4}

Let $\lambda>0$. Define the parabolic scaling
\[
u^\lambda(x,t):=\lambda^{-2}u(\lambda x,\lambda^2 t),\quad Q(x,t):=-\tau t+\frac{1}{2}x'Ax,
\] 
where the matrix $A$ and constant $\tau$ are as specified in \eqref{asymptotic}. To obtain a refined description of $u$, we establish the following convergence properties of $u^\lambda$.

\begin{lemma}\label{gradient holder}
$u^\lambda\rightarrow Q$ in $C^0_{loc}\left(\mathbb{R}^{n+1}_-\right)$ as $\lambda\rightarrow\infty$. Moreover, for every $\alpha\in(0,1)$, $D_xu^\lambda\rightarrow D_xQ$ in $C^{\alpha,\frac{\alpha}{2}}_{loc}\left(\mathbb{R}^{n+1}_-\right)$ along a subsequence as $\lambda\rightarrow\infty$.
\end{lemma}

\begin{proof}
By Proposition \ref{prop0}, we have for $\lambda^2\left(|x|^2-t\right)\ge C$ that
\[
\begin{aligned}
\left|u^\lambda(x,t)-Q(x,t)\right|&=\lambda^{-2}\left|u\left(\lambda x,\lambda^2t\right)-Q\left(\lambda x,\lambda^2t\right)\right|\\&\le C\lambda^{-\frac{\epsilon}{2}}\left(|x|^2-t\right)^{\frac{4-\epsilon}{4}}.
\end{aligned}
\] 
On the other hand for $\lambda^2\left(|x|^2-t\right)\le C$, it is obvious that
\[
\left|u^\lambda(x,t)\right|\le \lambda^{-2}\sup\limits_{E_{C^{1/2}}}|u|,\quad\left|Q(x,t)\right|\le\lambda^{-2}\sup\limits_{E_{C^{1/2}}}|Q|.
\]
Thus we conclude that $u^\lambda\rightarrow Q$ in $C^0_{loc}\left(\mathbb{R}^{n+1}_-\right)$ as $\lambda\rightarrow\infty$. 

Now we prove the second statement. Proposition \ref{bdd} gives that $\left|D_x^2u^\lambda\right|\le C$, which means
\[
\left|D_xu^\lambda(x_1,t)-D_xu^\lambda(x_2,t)\right|\le C|x_1-x_2|
\] 
for any $x_1,x_2\in\mathbb{R}^n$ and $t\in\left(-\infty,0\right]$. It follows from \eqref{u_t} that for any $x\in \mathbb{R}^n$ and $t_1,t_2\in\left(-\infty,0\right]$
\[
\left|u^\lambda(x_1,t_1)-u^\lambda(x_1,t_2)\right|\le m_2|t_1-t_2|.
\]
Applying Lemma 3.1 of Chapter 2 in \cite{lsu68}, we have for every $\alpha'\in(0,1)$ that
\[
\left|D_xu^\lambda(x,t_1)-D_xu^\lambda(x,t_2)\right|\le C(K)|t_1-t_2|^{\frac{\alpha'}{2}}
\]
for every bounded set $K$ in $\mathbb{R}^{n+1}_-$ and $(x,t_1),(x,t_2)\in K$, where $C(K)$ is a positive constant depending on $K$. Thus for every $\alpha''\in\left(0,\alpha'\right)$, there is
\[
\left|D_xu^\lambda(x_1,t_1)-D_xu^\lambda(x_2,t_2)\right|\le C(K)\left(|x_1-x_2|^2+|t_1-t_2|\right)^{\frac{\alpha''}{2}}.
\]
By Arzela-Ascoli theorem and $u^\lambda\rightarrow Q$ in $C^0_{loc}\left(\mathbb{R}^{n+1}_-\right)$ as $\lambda\rightarrow\infty$,  we get that for every $\alpha\in(0,1)$ 
\[
D_xu^\lambda\rightarrow D_xQ \quad \text{in} \quad C^{\alpha,\frac{\alpha}{2}}_{loc}\left(\mathbb{R}^{n+1}_-\right)
\]
along a subsequence as $\lambda\rightarrow\infty$.
\end{proof}

We will need the following density lemma derived from Theorem 7.37 in \cite{l96}. Introduce the notations
\[
K_-:=B_1\times(-5,-4),\quad K_+:=B_1\times(-1,0).
\]

\begin{lemma}\label{density}
Let $\Omega:=B_{4}\times(-16,0)$, $0<\lambda\le \Lambda<\infty$,  
\[
\lambda I\le \left(a_{ij}(x,t)\right)\le \Lambda I \quad \text{in}\quad \Omega,
\]
$u\in C^{2,1}\left(\Omega\right)$ satisfy
\[
u_t-a_{ij}D_{x_ix_j}u\ge0,\quad u\ge0 \quad \text{in}\quad \Omega,
\]
and
\[
\frac{\left|\left\{u\ge1\right\}\cap K_-\right|}{\left|K_-\right|}\ge \mu>0.
\]
Then there exists some $C=C(n,\lambda,\Lambda,\mu)>0$, such that $u\ge C^{-1}$ in $K_+$.
\end{lemma}

The following proposition establishes the precise supremum of $\Delta_e^2u$. Recall that $A$ is the symmetric positive definite $n\times n$ matrix in \eqref{asymptotic}.
\begin{prop}\label{deltax}
For every $e\in E$, we have the identity
\[
\sup\limits_{\mathbb{R}^{n+1}_-}\Delta_e^2u=\frac{e'Ae}{|e|^2}.
\]
\end{prop}

\begin{proof}
Let $\beta:=\frac{e'Ae}{|e|^2}$ and $\hat{e}:=\lambda^{-1}e$. 
We claim that 
\begin{equation}\label{claim1}
\lim_{\lambda\rightarrow\infty}\int_{K_-}\Delta_{\hat{e}}^2u^{\lambda} dxdt=\int_{K_-}\beta dxdt=\beta\left|K_-\right|.
\end{equation}
To prove the claim, we first observe that
\[
\begin{aligned}
&\quad\int_{K_-}\Delta_{\hat{e}}^2u^{\lambda} dxdt\\
&=\frac{1}{\left|\hat{e}\right|^2}\int_{-5}^{-4}\int_{B_1}\int_0^1\frac{d}{ds}\left[u^{\lambda}\left(x+s\hat{e},t\right)+u^{\lambda}\left(x-s\hat{e},t\right)\right]dsdxdt\\
&=\frac{1}{\left|\hat{e}\right|^2}\int_{-5}^{-4}\int_0^1\left( \int_{B_1\left(s\hat{e}\right)\setminus B_1\left(-s\hat{e}\right)}D_xu^{\lambda}\cdot\hat{e}dx - \int_{B_1\left(-s\hat{e}\right)\setminus B_1\left(s\hat{e}\right)}D_xu^{\lambda}\cdot\hat{e}dx \right)dsdt.
\end{aligned}
\]
Similarly we can get that
\[
\begin{aligned}
&\quad\int_{K_-}\beta dxdt=\int_{K_-}\Delta^2_{\hat{e}}Qdxdt\\
&=\frac{1}{\left|\hat{e}\right|^2}\int_{-5}^{-4}\int_0^1\left( \int_{B_1\left(s\hat{e}\right)\setminus B_1\left(-s\hat{e}\right)}D_xQ\cdot\hat{e}dx - \int_{B_1\left(-s\hat{e}\right)\setminus B_1\left(s\hat{e}\right)}D_xQ\cdot\hat{e}dx \right)dsdt.
\end{aligned}
\]
Lemma \ref{gradient holder} yields the convergence
\[
\begin{aligned}
&\left|\int_{K_-}\Delta_{\hat{e}}^2u^{\lambda} dxdt - \int_{K_-}\beta dxdt \right|\\
&\le\frac{1}{|\hat{e}|}\int_{-5}^{-4}\int_0^1  \int_{B_1\left(s\hat{e}\right)\setminus B_1\left(-s\hat{e}\right)}\left|D_xu^{\lambda}-D_xQ\right|dxdsdt\\
&\quad + \frac{1}{|\hat{e}|}\int_{-5}^{-4}\int_0^1  \int_{B_1\left(-s\hat{e}\right)\setminus B_1\left(s\hat{e}\right)}\left|D_xu^{\lambda}-D_xQ\right|dxdsdt\rightarrow0
\end{aligned}
\]
as $\lambda\rightarrow\infty$. We have proved the claim. 

Now we prove the proposition. Denote 
\[
\alpha:=\sup\limits_{\mathbb{R}^{n+1}_-}\Delta_e^2u.
\]
We have that $\alpha<\infty$ by Lemma \ref{bdd2}. The strict convexity of $u$ implies the $\lambda$-uniform bound $
0<\Delta_{\hat{e}}^2u^\lambda\le \alpha$. Hence it follows from \eqref{claim1} that $\alpha\ge\beta$. For contradiction, assume $\alpha>\beta$. \eqref{claim1} implies the measure estimate
\[
\limsup_{\lambda\rightarrow\infty}\left( \frac{\alpha+\beta}{2}\left|\left\{ \Delta_{\hat{e}}^2u^{\lambda}\ge \frac{\alpha+\beta}{2} \right\}\cap K_- \right| \right)\le\lim_{\lambda\rightarrow\infty}\int_{K_-}\Delta_{\hat{e}}^2u^{\lambda} dxdt=\beta\left|K_-\right|.
\]
Therefore for $\lambda$ large
\[
\frac{\left|\left\{ \Delta_{\hat{e}}^2u^{\lambda}\ge \frac{\alpha+\beta}{2} \right\}\cap K_- \right|}{\left|K_-\right|}\le\frac{2\beta}{\alpha+\beta}.
\]
Equivalently, we obtain the complementary measure
\[
\frac{\left|\left\{\frac{2\left(\alpha-\Delta_{\hat{e}}^2u^{\lambda}\right)}{\alpha-\beta}\ge1\right\}\cap K_-\right|}{\left|K_-\right|}\ge \frac{\alpha-\beta}{\alpha+\beta}.
\] 
Applying Lemma \ref{concave} and \ref{density} yields the pointwise bound
\[
\frac{2\left(\alpha-\Delta_{\hat{e}}^2u^{\lambda}\right)}{\alpha-\beta}\ge C^{-1}\quad \text{in} \quad K_+, 
\]
where $C>0$ is a constant. Consequently
\[
\Delta_{\hat{e}}^2u^{\lambda}\le \alpha-\frac{\alpha-\beta}{2C} \quad \text{in}\quad K_+.
\]
This leads to the contradiction
\[
\alpha=\sup\limits_{\mathbb{R}^{n+1}_-}\Delta_e^2u=\lim_{\lambda\rightarrow\infty}\sup\limits_{K_+}\Delta_{\hat{e}}^2u^{\lambda}\le \alpha-\frac{\alpha-\beta}{2C}<\alpha.
\]
Thus $\alpha=\beta$. The proposition has been proved.
\end{proof}

Define 
\[
\Delta_k^tu(x,t):=\frac{u(x,t)-u(x,t-k)}{k}.
\]
The concavity of $F$ directly yields the following pair of equations for $\Delta_k^tu$.

\begin{lemma}\label{concavity'}
The difference quotient $\Delta_k^tu$ satisfies 
\begin{equation}\label{tsubsolution}
\frac{1}{u_t(x,t-k)}D_t\Delta_k^tu+u^{ij}(x,t-k)D_{x_ix_j}\Delta_k^tu\ge0
\end{equation}
and
\begin{equation}\label{tsupersolution}
\frac{1}{u_t(x,t)}D_t\Delta_k^tu+u^{ij}(x,t)D_{x_ix_j}\Delta_k^tu\le0
\end{equation}
for every $(x,t)\in\mathbb{R}^{n+1}_-$ and $k\in\mathbb{N}^+$.
\end{lemma}

We give the exact value of $\Delta_k^tu$. Recall that $\tau$ is the constant in \eqref{asymptotic}.
\begin{prop}\label{deltat}
For every $(x,t)\in\mathbb{R}^{n+1}_-$ and $k\in\mathbb{N}^+$, we have the identity
\[
\Delta_k^tu=-\tau.
\]
\end{prop}

\begin{proof}
Denote $k_\lambda:=\frac{k}{\lambda^2}$. We claim that 
\begin{equation}\label{tclaim}
\lim_{\lambda\rightarrow\infty}\int_{K_-}\Delta_{k_\lambda}^tu^\lambda dxdt=\int_{K_-}(-\tau)dxdt=-\tau\left|K_-\right|.
\end{equation}
Actually
\[
\begin{aligned}
\int_{K_-}\Delta_{k_\lambda}^tu^\lambda dxdt
&=\int_{B_1}\int_{-5}^{-4}\frac{u^\lambda(x,t)-u^\lambda(x,t-k_\lambda)}{k_\lambda}dxdt\\
&=k_\lambda^{-1}\int_{B_1}\left(\int_{-5}^{-4}u^\lambda(x,t)dt-\int_{-5-k_\lambda}^{-4-k_\lambda}u^\lambda(x,t)dt\right)dx\\
&=k_\lambda^{-1}\int_{B_1}\left( \int^{-4}_{-4-k_\lambda}u^\lambda(x,t)dt-\int^{-5}_{-5-k_\lambda}u^\lambda(x,t)dt \right)dx.
\end{aligned}
\]	
We also have that
\[
\begin{aligned}
\int_{K_-}(-\tau)dxdt=k_\lambda^{-1}\int_{B_1}\left( \int^{-4}_{-4-k_\lambda}Q(x,t)dt-\int^{-5}_{-5-k_\lambda}Q(x,t)dt \right)dx.
\end{aligned}
\]
It follows from Lemma \ref{gradient holder} that
\[
\begin{aligned}
&\left|\int_{K_-}\Delta_{k_\lambda}^tu^\lambda dxdt-\int_{K_-}(-\tau)dxdt\right|\\
&\le k_\lambda^{-1}\int_{B_1}\left( \int^{-4}_{-4-k_\lambda}\left|u^\lambda-Q\right|(x,t)dt+\int^{-5}_{-5-k_\lambda}\left|u^\lambda-Q\right|(x,t)dt \right)dx\rightarrow0
\end{aligned}
\]
as $\lambda\rightarrow\infty$. The claim has been proved. 

To prove $\Delta_k^tu=-\tau$, we just need to get that
\begin{equation}\label{twoinequalities}
\sup_{\mathbb{R}^{n+1}_-}\Delta_k^tu=-\tau=\inf_{\mathbb{R}^{n+1}_-}\Delta_k^tu. 
\end{equation}

Now we give the proof of the first equality in \eqref{twoinequalities}. Denote 
\[
v:=u+m_2t,\quad v^\lambda(x,t):=\lambda^{-2}v\left(\lambda x,\lambda^2 t\right),\quad \gamma:=\sup\limits_{\mathbb{R}^{n+1}_-}\Delta_k^tv.
\]
From \eqref{u_t} we know that 
\[
-m_2\le\inf_{\mathbb{R}^{n+1}_-}\Delta_k^tu\le -\tau \le\sup_{\mathbb{R}^{n+1}_-}\Delta_k^tu\le -m_1\quad \text{in}\quad\mathbb{R}^{n+1}_-,
\]
where the last inequality gives that $\gamma<\infty$ and the first one implies that
\begin{equation}\label{ge01}
\Delta_{k_\lambda}^tv^\lambda=\Delta_{k_\lambda}^tu^\lambda+m_2\ge0.
\end{equation}
\eqref{tclaim} implies that
\[
\lim_{\lambda\rightarrow\infty}\int_{K_-}\Delta_{k_\lambda}^tv^\lambda dxdt=\left(m_2-\tau\right)\left|K_-\right|.
\]
Hence we know that $\gamma\ge m_2-\tau$. Suppose $\gamma>m_2-\tau$, then by \eqref{ge01}
\[
\begin{aligned}
&\limsup_{\lambda\rightarrow\infty}\left( \frac{\gamma+m_2-\tau}{2}\left|\left\{ \Delta_{k_\lambda}^tv^{\lambda}\ge \frac{\gamma+m_2-\tau}{2} \right\}\cap K_- \right| \right)\\
&\le \lim_{\lambda\rightarrow\infty}\int_{K_-}\Delta_{k_\lambda}^tv^{\lambda} dxdt=\left(m_2-\tau\right)\left|K_-\right|.
\end{aligned}
\]
Thus for $\lambda$ large
\[
\frac{\left|\left\{ \Delta_{k_\lambda}^tv^{\lambda}\le \frac{\gamma+m_2-\tau}{2} \right\}\cap K_- \right|}{\left|K_-\right|}\le\frac{2\left(m_2-\tau\right)}{\gamma+m_2-\tau},
\]
which means that
\[
\frac{\left|\left\{\frac{2\left(\gamma-\Delta_{k_\lambda}^tv^{\lambda}\right)}{\gamma-m_2+\tau}\ge1\right\}\cap K_-\right|}{\left|K_-\right|}\ge \frac{\gamma-m_2+\tau}{\gamma+m_2-\tau}.
\] 
By Lemma \ref{density} and \eqref{tsubsolution} we arrive at 
\[
\frac{2\left(\gamma-\Delta_{k_\lambda}^tv^{\lambda}\right)}{\gamma-m_2+\tau}\ge C^{-1}\quad \text{in} \quad K_+
\]
for a positive constant $C>0$. In other words
\[
\Delta_{k_\lambda}^tv^{\lambda}\le \gamma-\frac{\gamma-m_2+\tau}{2C} \quad \text{in}\quad K_+.
\]
It follows that
\[
\gamma=\sup\limits_{\mathbb{R}^{n+1}_-}\Delta_{k}^tv=\lim_{\lambda\rightarrow\infty}\sup\limits_{K_+}\Delta_{k_\lambda}^tv^{\lambda}\le \gamma-\frac{\gamma-m_2+\tau}{2C}<\gamma,
\]
which is a contradiction. Therefore $\gamma=m_2-\tau$.

For the second equality in \eqref{twoinequalities}, we present its proof's outline. We name
\[
\eta:=\sup\limits_{\mathbb{R}^{n+1}_-}\Delta_k^t(-u).
\] 
It is easy to see that 
\begin{equation}\label{ge02}
\Delta_{k_\lambda}^t\left(-u^\lambda\right)\ge0
\end{equation}
by \eqref{u_t}. It follows from \eqref{tclaim} that
\[
\lim_{\lambda\rightarrow\infty}\int_{K_-}\Delta_{k_\lambda}^t\left(-u^\lambda\right) dxdt=\tau\left|K_-\right|.
\]
Combined with \eqref{ge02}, there is
\[
\limsup_{\lambda\rightarrow\infty}\left( \frac{\eta+\tau}{2}\left|\left\{ \Delta_{k_\lambda}^t\left(-u^\lambda\right)\ge \frac{\eta+\tau}{2} \right\}\cap K_- \right| \right)\le \tau\left|K_-\right|.
\]
Thus for $\lambda$ large
\[
\frac{\left|\left\{\frac{2\left(\eta-\Delta_{k_\lambda}^t\left(-u^\lambda\right)\right)}{\eta-\tau}\ge1\right\}\cap K_-\right|}{\left|K_-\right|}\ge \frac{\eta-\tau}{\eta+\tau}.
\]
Lemma \ref{density} and \eqref{tsupersolution} imply that
\[
\Delta_{k_\lambda}^t\left(-u^\lambda\right)\le \eta-\frac{\eta-\tau}{2C} \quad \text{in}\quad K_+,
\]
which brings about a contradiction
\[
\eta=\sup\limits_{\mathbb{R}^{n+1}_-}\Delta_{k}^t(-u)=\lim_{\lambda\rightarrow\infty}\sup\limits_{K_+}\Delta_{k_\lambda}^t\left(-u^\lambda\right)\le \eta-\frac{\eta-\tau}{2C}<\eta.
\]
This completes the proof of the second equality in \eqref{twoinequalities}.
\end{proof}

\quad

\section{Proof of Theorem \ref{mainthm1}}\label{5}
In this section, we establish Theorem \ref{mainthm1}, demonstrating that $u-Q$ consists of an affine function and a periodic function.

Choose $b\in\mathbb{R}^n$ so that the function
\[
\tilde{w}(x,t):=u(x,t)-Q(x,t)-b\cdot x
\]
satisfies $\tilde{w}(e_k,0)=\tilde{w}(-e_k,0)$ for $1\le k\le n$. Clearly $\tilde{w}\left(0_n,0\right)=0$. Due to Proposition \ref{deltax}, there is
\[
\Delta_{e_k}^2\tilde{w}\le0
\] 
for $1\le k\le n$. This inequality, combined with Lemma A.3 in \cite{cl04}, yields the fundamental estimate
\begin{equation}\label{w<=0}
\tilde{w}(je_k,0)\le0\quad \text{for}\quad 1\le k\le n \quad \text{and}\quad j\in\mathbb{Z}.
\end{equation}
For reader's convenience, we state it as follows:
\begin{lemma}[Lemma A.3 in \cite{cl04}]
Let $g\in C^0\left(\mathbb{R}\right)$, $g(0)=0$, $g(1)=g(-1)$, and $\Delta_1^2g(x)\le0$ for all $x\in\mathbb{R}$. Then
\[
g(m+1)\le g(m),\quad g(-m-1)\le g(-m)
\]
for all non-negative integer $m$. Consequently, $g(m)\le 0$ for all integer $m$.
\end{lemma}

By Proposition \ref{deltat}, we arrive at
\[
\Delta_1^t\tilde{w}=0 \quad \text{in} \quad \mathbb{R}^{n+1}_-,
\] 
which means that $\tilde{w}$ has 1-periodicity with respect to $t$ variable. There is
\[
-(\tilde{w}+Q)_t\det D_x^2(\tilde{w}+Q)=f\quad\text{in}\quad \mathbb{R}^{n+1}_-,
\]
together with uniform bounds
\begin{equation}\label{twouniformbounds}
m_1\le-(\tilde{w}+Q)_t\le m_2,\quad C^{-1}I\le D_x^2(\tilde{w}+Q)\le CI\quad\text{in}\quad \mathbb{R}^{n+1}_-.
\end{equation}

Now we construct a periodic function. Applying Theorem 0.2 in \cite{cl04}, there exists $\xi_1\in C^2\left(\mathbb{R}^n\right)$, which is the unique solution of 
\[
\det\left(A+D_x^2\xi_1\right)=\det A\cdot f_1\quad\text{in}\quad\mathbb{R}^n
\]
satisfying
\[
A+D_x^2\xi_1>0\quad \text{in}\quad\mathbb{R}^n,
\]
\[
\xi_1\left(x+e_i\right)=\xi_1(x)\quad \text{for} \quad x\in\mathbb{R}^n \quad \text{and}\quad 1\le i\le n,
\]
and $\xi_1\left(0_n\right)=0$. Then there exists $C>0$ such that 
\begin{equation}\label{xuniformbounds}
C^{-1}I\le D_x^2\left(\xi_1+Q\right)\le CI\quad\text{in}\quad\mathbb{R}^n.
\end{equation}
For the temporal component, define
\[
\xi_2(t):=-\tau\int_0^t\left(f_2(s)-1\right)ds,
\]
yielding the derivative bound
\begin{equation}\label{tuniformbounds}
-(\xi_2+Q)_t=\tau f_2\in\left[\tau\inf_{\left(-\infty,0\right]}f_2,\tau\sup_{\left(-\infty,0\right]}f_2\right].
\end{equation}
Thus $v:=\xi_1+\xi_2\in C^{2,1}\left(\mathbb{R}^{n+1}_-\right)$ and solves
\[
-\left(v+Q\right)_t\det D_x^2\left(v+Q\right)=f\quad\text{in}\quad\mathbb{R}^{n+1}_-.
\]
Additionally
\begin{equation}\label{v=0}
v\left(je_k,0\right)=\xi_1\left(je_k\right)=0\quad\text{for}\quad 1\le k\le n \quad \text{and}\quad j\in\mathbb{Z}.
\end{equation}

Define the difference $h:=\tilde{w}-v$. According to \eqref{w<=0} and \eqref{v=0}, there is
\begin{equation}\label{h<=0}
h\left(je_k,0\right)\le0\quad\text{for}\quad 1\le k\le n \quad \text{and}\quad j\in\mathbb{Z}.
\end{equation}
$h$ satisfies linear parabolic equation 
\[
h_t-a_{ij}D_{x_ix_j}h=0\quad \text{in}\quad \mathbb{R}^{n+1}_-, 
\]
where
\[
a_{ij}(x,t):=\frac{\int_0^1\left[sD_x^2(w+Q)+(1-s)D_x^2(v+Q)\right]^{ij}ds}{\int_0^1\left[-s(w+Q)_t+(s-1)(v+Q)_t\right]^{-1}ds}.
\]
Crucially by \eqref{twouniformbounds}, \eqref{xuniformbounds} and \eqref{tuniformbounds}, the equation is uniformly parabolic: 
\[
C^{-1}I\le \left(a_{ij}(x,t)\right)\le CI \quad \text{in}\quad \mathbb{R}^{n+1}_-.
\]
The temporal derivative is also bounded:
\begin{equation}\label{h_t}
-C_0\le h_t\le C_0,
\end{equation}
where $C_0>0$ is a constant. From the periodicity of $v$ and $w$, $h$ is periodic in $t$. 

First, we will prove that $h$ has a finite supremum.
\begin{prop}\label{h}
$\sup\limits_{\mathbb{R}^{n+1}_-}h<\infty$.
\end{prop}

\begin{proof}
Let $M_i:=\sup\limits_{\left[-i,i\right]^n\times\left[-i^2,0\right]}h$. Suppose the contrary. Then $\lim\limits_{i\rightarrow\infty}M_i=\infty$. 

\textbf{Step 1: Obtain $M_i$'s iteration relation from $\Delta_e^2h\le0$.} We claim that
\begin{equation}\label{M_i}
\sup_{\left[-2^i,2^i\right]^n}h(x,0)\le 4\sup_{\left[-2^{i-1},2^{i-1}\right]^n}h(x,0)+C_1,\quad i\ge3
\end{equation}
for $C_1>0$ independent of $i$. Actually, there is a constant $C>0$ such that 
\[
\left|h(x,0)\right|\le C \quad \text{for} \quad x\in[-1,1]^n.
\]
For $m\in\mathbb{N}\setminus\left\{0\right\}$ and $x=\left(x_1,\cdots,x_n\right)\in [-m,m]^n$, we denote $[x_k]$ as the integer part of $x_k$ and define
\[
\epsilon_k:=
\left\{\begin{array}{ll}
1 & \text { if } x_k \text{ is odd}, \\
0 & \text { if } x_k \text{ is even}.
\end{array}\right.
\]
Then by Proposition \ref{deltax}, there is $\Delta_e^2h(x,0)=\Delta_e^2w(x,0)\le0$ for every $x\in\mathbb{R}^n$ and $e\in E$. It follows that for $x\in[-m,m]^n$
\[
h(x,0)+h\left( x-\sum_{k=1}^n\left([x_k]+\epsilon_k\right)e_k,0 \right)\le 2h\left( x-\sum_{k=1}^n\frac{[x_k]+\epsilon_k}{2}e_k,0 \right).
\]
Since 
\[
x-\sum_{k=1}^n\left([x_k]+\epsilon_k\right)e_k\in[-1,1]^n
\] 
and
\[
x-\sum_{k=1}^n\frac{[x_k]+\epsilon_k}{2}e_k\in\left[-\left[\frac{m+1}{2}\right]-1,\left[\frac{m+1}{2}\right]+1\right]^n,
\]
we get that
\[
h(x,0)\le 2\sup_{\left[-\left[\frac{m+1}{2}\right]-1,\left[\frac{m+1}{2}\right]+1\right]^n}h(x,0)+C_2.
\]
It follows that
\[
\begin{aligned}
\sup_{\left[-m,m\right]^n}h(x,0) &\le 2\sup_{\left[-\left[\frac{m+1}{2}\right]-1,\left[\frac{m+1}{2}\right]+1\right]^n}h(x,0)+C_2\\
&\le 4\sup_{\left[-\left[\frac{m+1}{4}\right]-2,\left[\frac{m+1}{4}\right]+2\right]^n}h(x,0)+3C_2.
\end{aligned}
\]
Take $m=2^i$ and $i\ge3$, we have proved the claim for $C_1:=3C_2$. 

By \eqref{h_t} and the periodicity of $h$ in $t$, we obtain the iteration relation
\begin{equation}\label{iterate}
M_{2^i}\le\sup_{\left[-2^i,2^i\right]^n}h(x,0)+C_0\le 4\sup_{\left[-2^{i-1},2^{i-1}\right]^n}h(x,0) +C_0+C_1\le 4M_{2^{i-1}}+C.
\end{equation}

\textbf{Step 2: Scare $h$ to get a limit funtion $H$. }Define the rescaled functions
\[
H_i(x,t):=\frac{h\left(2^ix,2^{2i}t\right)}{M_{2^i}}\quad\text{for}\quad(x,t)\in\left[-1,1\right]^n\times\left(-\infty,0\right].
\]
$H_i$ has a periodicity of $2^{-2i}$ in $t$ and satisfies a uniformly parabolic equation. We can easily derive that $H_i\left(0_n,0\right)=0$ and $H_i\le 1$ in $\left[-1,1\right]^n\times\left(-\infty,0\right]$. There is also
\[
H_{i}\left(\pm\frac{1}{2}e_k,0\right)=\frac{h\left(\pm2^{i-1}e_k,0\right)}{M_{2^i}}\le0,\quad 1\le k\le n
\]
from \eqref{h<=0}. From the iteration inequality \eqref{iterate}, we have that for large $i$
\[
\sup_{\left[-\frac{1}{2},\frac{1}{2}\right]^n\times\left[-\frac{1}{4},0\right]}H_i=\frac{M_{2^{i-1}}}{M_{2^i}}\ge\frac{M_{2^i}-6C}{4M_{2^i}}\ge\frac{1}{8}.
\]
Applying the Harnack inequality in \cite{l96} (Corollary 7.42) and periodicity of $H_i$
\[
\sup_{\left[-\frac{8}{9},\frac{8}{9}\right]^n\times\left[-1,0\right]}\left(1-H_i\right)\le C\left(1-H_i\left(0_n,0\right)\right)=C,
\]
which yields that $1-C\le H_i\le 1$ in $\left[-\frac{8}{9},\frac{8}{9}\right]^n\times\left[-1,0\right]$. By the Hölder estimate in \cite{l96} (Theorem 7.41), there exists $\beta\in(0,1)$, such that
\[
\left\|H_i\right\|_{C^{\beta,\frac{\beta}{2}}\left(\left[-\frac{3}{4},\frac{3}{4}\right]^n\times\left[-\frac{3}{4},0\right]\right)}\le C.
\]
Take $\alpha\in(0,\beta)$, and it follows from Arzela-Ascoli theorem that there exists $H\in C^{\alpha,\frac{\alpha}{2}}\left(\left[-\frac{3}{4},\frac{3}{4}\right]^n\times\left[-\frac{3}{4},0\right]\right)$, such that
\[
H_i\rightarrow H\quad\text{in}\quad C^{\alpha,\frac{\alpha}{2}}\left(\left[-\frac{3}{4},\frac{3}{4}\right]^n\times\left[-\frac{3}{4},0\right]\right)
\]
along a subsequence as $i\rightarrow\infty$. 

\textbf{Step 3: Claim that $H=H(x)$ and $H$ is concave. }To prove that $H$ depends only on $x$, we notice that 
\begin{equation}\label{almostconst}
\Delta_{2^{-2i}k}^tH_i=0
\end{equation}
inherited from $\Delta_k^th=0$. Fix $t_1,t_2\in\left[-\frac{3}{4},0\right]$. For any $i\ge3$ , there exists $k_i\in\mathbb{N}$, such that
\[
\left|t_2-\left(t_1-2^{-2i}k_i\right)\right|\le 2^{-2i}.
\]
It follows from \eqref{almostconst} that
\[
H_i\left(x,t_1\right)=H_i\left(x,t_1-2^{-2i}k_i\right) \quad\text{for}\quad x\in \left[-\frac{3}{4},\frac{3}{4}\right]^n.
\]
Let $i\rightarrow\infty$, we have $H\left(x,t_1\right)=H\left(x,t_2\right)$. The arbitrariness of $t_1,t_2$ gives the claim.

For the second claim, we have that 
\begin{equation}\label{almostconcavity}
\Delta_{2^{-i}e}^2H_i\le0
\end{equation}
induced from $\Delta_e^2h\le0$. Fix $x,y$ such that $x,x-y,x+y\in \left[-\frac{3}{4},\frac{3}{4}\right]^n$. For any $i\ge3$, there exists $e^i\in E$ with $|e^i|\le 2^i\sqrt{n}$, such that $\left|y-2^{-i}e^i\right|\le 2^{-i}\sqrt{n}$, $x+2^{-i}e^i,x-2^{-i}e^i\in \left[-\frac{3}{4},\frac{3}{4}\right]^n$. \eqref{almostconcavity} implies that
\[
H_i\left(x+2^{-i}e^i\right)+H_i\left(x-2^{-i}e^i\right)\le 2H_i\left(x\right).
\]
Let $i\rightarrow\infty$, we have that
\[
H(x+y)+H(x-y)\le 2H(x),
\]
which implies the concavity of $H$, which indicates the second claim.

\textbf{Step 4: Derive a contradiction by the Harnack inequality.} From the properties of $H_i$, the function $H$ satisfies
\[
H\left(0_n\right)=0,\quad \sup_{\left[-\frac{1}{2},\frac{1}{2}\right]^n}H\ge\frac{1}{8},\quad H\left(\pm\frac{1}{2}e_k\right)\le0\quad \text{for} \quad 1\le k\le n.
\]
Since $H$ is concave and $H\left(0_n\right)=0$, there exists a linear function $l$, such that $l\ge H$ in $\left[-\frac{3}{4},\frac{3}{4}\right]^n$ and $l\left(0_n\right)=0$. For any $\epsilon>0$, there exists $i'$, such that
\[
\left\| H-H_{i'} \right\|_{C^{\alpha,\frac{\alpha}{2}}\left(\left[-\frac{3}{4},\frac{3}{4}\right]^n\times\left[-\frac{3}{4},0\right]\right)}< \epsilon.
\]
Then by Harnack inequality
\[
\begin{aligned}
\sup_{\left[-\frac{3}{4},\frac{3}{4}\right]^n\times\left[-\frac{3}{4},0\right]}(l-H)&\le \sup_{\left[-\frac{3}{4},\frac{3}{4}\right]^n\times\left[-\frac{3}{4},0\right]}\left(l-H_{i'}+\epsilon\right)\\
&\le C\inf_{\left[-\frac{3}{4},\frac{3}{4}\right]^n\times\left[-\frac{3}{4},0\right]}\left(l-H_{i'}+\epsilon\right)\\
&\le C\left( \inf_{\left[-\frac{3}{4},\frac{3}{4}\right]^n\times\left[-\frac{3}{4},0\right]}\left(l-H\right) +2\epsilon \right)\\
&=2C\epsilon.
\end{aligned}
\]
Let $\epsilon\rightarrow0$, $\sup\limits_{\left[-\frac{3}{4},\frac{3}{4}\right]^n\times\left[-\frac{3}{4},0\right]}(l-H)\equiv0$, which means that $H\equiv l$. Because $H\left(0_n\right)=0$ and $H\left(\pm\frac{1}{2}e_k,0\right)\le0$ for $1\le k\le n$, we actually get that $H\equiv l\equiv0$ in $\left[-\frac{3}{4},\frac{3}{4}\right]^n$. This contradicts $\sup\limits_{\left[-\frac{1}{2},\frac{1}{2}\right]^n}H\ge\frac{1}{8}$! Thus $\sup\limits_{\mathbb{R}^{n+1}_-}h<\infty$.
\end{proof}

The combination of periodicity in the temporal variable t and uniform boundedness of $h$ implies that $h$ must be constant via the Harnack inequality.

\begin{proof}[proof of Theorem \ref{mainthm1}]
Denote $\gamma:=\sup\limits_{\mathbb{R}^{n+1}_-}h$. Then 
\[
\left(\gamma-h\right)_t-a_{ij}D_{x_ix_j}\left(\gamma-h\right)=0\quad \text{in}\quad \mathbb{R}^{n+1}_-.
\] 
For any $\epsilon>0$, there exists $\left(\overline{x},\overline{t}\right)\in \mathbb{R}^n\times[-1,0]$, such that $h\left(\overline{x},\overline{t}\right)>\gamma-\epsilon$. Denote $h_R(x,t):=h(Rx,R^2t)$, then
\[
\left(\gamma-h_R\right)_t-a_{ij}(Rx,R^2t)D_{x_ix_j}\left(\gamma-h_R\right)=0 \quad\text{in}\quad \mathbb{R}^{n+1}_-,
\]
which is still a uniformly parabolic equation with parabolic constants independent of $R$. For every $R$ such that $R\ge 2|\overline{x}|$. We have by the Harnack inequality that
\[
\begin{aligned}
\sup_{B_\frac{1}{2}\times[-3,-2]}\left(\gamma-h_R\right)\le C\left(\gamma-h_R\left(\frac{\overline{x}}{R},\frac{\overline{t}}{R^2}\right)\right)<C\epsilon.
\end{aligned}
\]
Let $R\rightarrow\infty$, then $h\ge\gamma-C\epsilon$ in $\mathbb{R}^{n+1}_-$. Let $\epsilon\rightarrow0$, there is $h\equiv\gamma$. Thus 
\[
\begin{aligned}
u(x,t)&=\tilde{w}(x,t)+Q(x,t)+b\cdot x=\gamma+v(x,t)-\tau t+\frac{1}{2}x'Ax+b\cdot x.
\end{aligned}
\]
\end{proof}

\noindent {\bf Funding:} J. Bao is supported by the National Key Research and Development Program of China (2020YFA0712904) and the National Natural Science Foundation of China (12371200).

\medskip

\bibliographystyle{plain}
\def\cprime{$'$}

\end{document}